\documentclass[12pt]{amsart}
\usepackage{amsmath, amsfonts, amssymb,amsthm, amscd}

\usepackage{pictexwd,dcpic}
\usepackage{graphicx}
\usepackage{epsf}
\usepackage{epsfig}
\usepackage{rlepsf}
\usepackage{psfrag}
\usepackage{color}
\usepackage{hyperref}
\newtheorem*{thnnumber}{Theorem}
\theoremstyle{definition}

\theoremstyle{remark}

\newtheoremstyle{thm}% name
  {12pt}%      Space above, empty = `usual value'
  {12pt}%      Space below
  {\itshape}% Body font
  {\parindent}%         Indent amount (empty = no indent, \parindent = para indent)
  {\scshape}% Thm head font
  {.}%        Punctuation after thm head
  {5pt}% Space after thm head: \newline = linebreak
  {}%         Thm head spec

\theoremstyle{thm}
\newtheorem*{T6.17}{Theorem \ref{gradientthm1}}
\newtheorem{thm}{Theorem}[section]

\newtheoremstyle{prop}% name
  {12pt}%      Space above, empty = `usual value'
  {12pt}%      Space below
  {\itshape}% Body font
  {\parindent}%         Indent amount (empty = no indent, \parindent = para indent)
  {\scshape}% Thm head font
  {.}%        Punctuation after thm head
  {5pt}% Space after thm head: \newline = linebreak
  {}%         Thm head spec

\theoremstyle{prop}
\newtheorem{prop}[thm]{Proposition}
\newtheoremstyle{lem}% name
  {12pt}%      Space above, empty = `usual value'
  {12pt}%      Space below
  {\itshape}% Body font
  {\parindent}%         Indent amount (empty = no indent, \parindent = para indent)
  {\scshape}% Thm head font
  {.}%        Punctuation after thm head
  {5pt}% Space after thm head: \newline = linebreak
  {}%         Thm head spec

\theoremstyle{lem}
\newtheorem*{L2.6}{Lemma \ref{lem2.4}}
\newtheorem{lem}[thm]{Lemma}

\newtheoremstyle{defn}% name
  {12pt}%      Space above, empty = `usual value'
  {12pt}%      Space below
  {\itshape}% Body font
  {\parindent}%         Indent amount (empty = no indent, \parindent = para indent)
  {\scshape}% Thm head font
  {.}%        Punctuation after thm head
  {5pt}% Space after thm head: \newline = linebreak
  {}%         Thm head spec

\theoremstyle{defn}
\newtheorem{defn}[thm]{Definition}

\newtheoremstyle{examp}% name
  {12pt}%      Space above, empty = `usual value'
  {12pt}%      Space below
  {}% Body font
   {\parindent}%         Indent amount (empty = no indent, \parindent = para indent)
  {\scshape}% Thm head font
  {.}%        Punctuation after thm head
  {5pt}% Space after thm head: \newline = linebreak
  {}%         Thm head spec

\theoremstyle{examp}

\newtheoremstyle{cor}% name
  {12pt}%      Space above, empty = `usual value'
  {12pt}%      Space below
  {\itshape}% Body font
  {\parindent}%         Indent amount (empty = no indent, \parindent = para indent)
  {\scshape}% Thm head font
  {.}%        Punctuation after thm head
  {5pt}% Space after thm head: \newline = linebreak
  {}%         Thm head spec

\theoremstyle{cor}

\newtheoremstyle{recipe}% name
  {12pt}%      Space above, empty = `usual value'
  {12pt}%      Space below
  {\itshape}% Body font
   {\parindent}%         Indent amount (empty = no indent, \parindent = para indent)
  {\scshape}% Thm head font
  {.}%        Punctuation after thm head
  {5pt}% Space after thm head: \newline = linebreak
  {}%         Thm head spec

\theoremstyle{recipe}

\newtheoremstyle{rem}% name
  {12pt}%      Space above, empty = `usual value'
  {12pt}%      Space below
  {}% Body font
   {\parindent}%         Indent amount (empty = no indent, \parindent = para indent)
  {\scshape}% Thm head font
  {.}%        Punctuation after thm head
  {5pt}% Space after thm head: \newline = linebreak
  {}%         Thm head spec

\theoremstyle{rem}
\newtheorem{rem}[thm]{Remark}

\newcommand{\bp}{\begin{proof}}
\newcommand{\ep}{\end{proof}}

\newcommand{\abs}[1]{\left\vert#1\right\vert}

\newcommand{\ssc}{\text{sc}}

\renewcommand{\epsilon}{\varepsilon}

\newcommand{\wh}{\widehat}

\newcommand{\wt}{\widetilde}
 \newcommand{\e}{\text{e}}

\newcommand{\ov}{\overline}

\newcommand{\ke}{K^{\text{e}}}
\newcommand{\kin}{K^{\text{in}}}
\newcommand{\st}{S_{\Theta}}
\newcommand{\sg}{S_{\Theta}^{\text{g}}}
\newcommand{\sba}{S_{\Theta}^{\text{b}}}
\newcommand{\sth}{\supp \Theta}

\newcommand{\muo}{\mu_{\omega}}

\newcommand{\cl}{{\mathcal L}}

\newcommand{\id}{\operatorname{id}}

\newcommand{\supp}{\operatorname{supp}}

\newcommand{\R}{{\mathbb R}}

\newcommand{\N}{{\mathbb N}}
\newcommand{\Q}{{\mathbb Q}}
\newcommand{\dif}{\text{Diff}_{\ssc}}
\def\abs#1{\mathopen|#1\mathclose|}

{\catcode`"=12 \gdef\hex{"}}

\mathchardef\laplace=\hex0001

\mathchardef\nabla=\hex0272

\makeatletter

\def\@@dalembert#1#2{\setbox0\hbox{$#1\mathrm I$}

  \vrule height\ht0 depth\z@ width.04\ht0

  \rlap{\vrule height\ht0 depth-.96\ht0 width.8\ht0}

  \vrule height.1\ht0 depth\z@ width.8\ht0

  \vrule height\ht0 depth\z@ width.1\ht0 }

\def\dalembert{\mathbin{\mathpalette\@@dalembert{}}\,}

\makeatother

\begin{document}

%\frontmatter
\title[Integration]{Integration Theory for  zero sets of polyfold Fredholm
sections} \maketitle  \begin{center}
%\textcolor{blue}{\bf DATE 11-5-2007}
\end{center}\vspace{0.5cm}\begin{center} H.
Hofer\footnote{Research partially supported by NSF grant
DMS-0603957.}, K. Wysocki \footnote{Research partially supported by
NSF grant DMS-0606588. } and E. Zehnder
 \footnote{Research partially supported by  TH-project.}
\end{center}

 \tableofcontents

\section{Introduction}
The polyfold Fredholm theory, developed  in
\cite{HWZ1,HWZ2,HWZ3}, extends the classical Fredholm theory. An
outline of the theory and its applications is given in \cite{Hofer}.
One of the current main applications is the Symplectic Field Theory
(SFT)  initiated in  \cite{EGH}. It is   a  theory of symplectic invariants generalizing
the Gromov-Witten theory which started with Gromov's seminal paper \cite{G}. The SFT
has a rich algebraic structure. In order to make the algebraic formalism proposed in
\cite{EGH} rigorous one needs an integration theory for differential
forms over the solution spaces. This is not straightforward since,
in general, even for the choice of generic geometric data, there are
numerous compactness and transversality issues, so that a priori the
solution sets are quite bad sets. Fortunately there is a good
compactification of  the solution spaces for which we refer to  \cite{BEHWZ}. This
compactification is in part responsible for the wealth of algebraic
structure and it allows to apply the polyfold Fredholm theory. As
far as the transversality issues are concerned the polyfold theory
constructs a ``larger universe" in which the problems can be
 perturbed in an abstract way in order  to achieve transversality. Since,  in general,  we are dealing simultaneously
 with infinitely many interrelated Fredholm
problems,  the perturbations are not independent of each other.
Moreover,  we are dealing with equivalence classes of Fredholm
problems leading to local symmetries (since two different problems
might be isomorphic in several different ways). All these structures
and symmetries have to be preserved  under the  abstract
perturbations of the problems. It turns out that  single-valued
perturbations are not sufficient and multi-valued perturbations
compatible with the symmetries have to be used. Consequently the
solution sets will,  in general,  not be manifolds. Still, generically,
they carry enough structure so that one can define an integration of
differential forms over them. In some sense this is an integration
over a space obtained by dividing out from a category its morphisms,
i.e., an integration over isomorphism classes of objects. We call
this version of integration  ``branched integration". As we shall
see, Stokes' theorem is still valid. Our paper is related to recent
work by McDuff  \cite{Mc}, where she studies similar questions which
were also in part motivated by the polyfold theory. It is also
related to  the work of  Cieliebak, Mundet i Riera and Salamon  \cite{CRS}, where solution sets of Fredholm problems
equivariant with respect to a global group action are studied. In
fact, we pick up some of the ideas  from  \cite{CRS} and bring them into a
more general context. The basic goal in \cite{Mc} is to give an
intrinsic structure to the solution sets which among other things
allows an integration theory. In our theory we can integrate on
solution sets which are obtained under less transversality
assumptions than in \cite{Mc}. This will be important since the compatibility of perturbations with  certain structures (f.e operations, see \cite{Hofer} and the upcoming \cite{HWZ6}) quite often constraint the level of achievable genericity. Let us note that we did not attempt
in this paper to give an intrinsic structure to the solution sets,
i.e.,  view them as abstract objects without reference to the ambient
polyfold. This  is not needed for defining invariants. On the other hand,
it is  interesting to understand the  intrinsic
structure of the solution sets.  This problem will  be studied in  a
follow-up paper where  structures of the solution sets are described
which are finer than those constructed in \cite{Mc}.

\subsection{Background Material}

The reader is expected to be  familiar with the polyfold
theory presented in \cite{HWZ1, HWZ2, HWZ3}. We
recall that an M-polyfold $X$ is a second countable paracompact
space equipped with an atlas consisting of charts
$(U,\varphi,(O,{\mathcal S}))$ whose transition maps are sc-smooth.
The chart maps $\varphi:U\subset X\to O$ are homeomorphisms onto
open subsets $O$ of splicing cores. They are the local models
(replacing the open
 subsets of Banach spaces in the manifold theory)
 which lead to the generalized differential  geometry
  presented in \cite{HWZ1}.  The M-polyfolds can have subpolyfolds. In the following an important role is played by the special class of {\bf  finite dimensional submanifolds}  with boundaries with corners introduced in Definition 4.20 of \cite{HWZ2} and, for the reader's convenience, recalled  in Appendix \ref{fdimsubm}.

The boundary $\partial X$ of an M-poyfold $X$ is defined using the degeneracy index
$d:X\to \N_0:=\N\cup \{0\}$  introduced in \cite{HWZ1}.  The map $d$  has the following interpretation. A point $x\in X$ satisfying $d(x)=0$ is an interior point.
The boundary  $\partial X$ is  defined as
$$\partial X=\{x\in X\vert \, d(x)\geq 1\}.$$  A point $x$ satisfying $d(x)=1$ is called a good boundary point. Points satisfying  $d(x)\geq 2$ are corner points and the integer $d(x)$ indicates the order of the corner.

We recall   the concept of an ep-groupoid as introduced in \cite{HWZ3}. We
begin with the notion of a  {\bf groupoid}. A groupoid  is   a small
category $\mathfrak{G}$ in which every morphism (or arrow)  is
invertible. We shall denote its set of objects by $G$ and its set of
morphisms by ${\bf G}$. There are the usual five structure  maps
$(s, t, m, u, i)$. Namely, the source and target maps $s,t:{\bf
G}\rightarrow G$ assign to every morphism, denoted by $g:x\to y$,
its source $s(g)=x$ and its target $t(g)=y$.  The associative
multiplication (or composition) map
$$m: {\bf G}{{_s}\times_t}{\bf G}\rightarrow {\bf G}, \quad m(h,g)=h\circ g$$
is defined on the fibered product
$$ {\bf G}{{_s}\times_t}{\bf G}=\{(h, g)\in {\bf G}\times {\bf G}\vert \, s(h)=t(g)\}.$$
For every object $x\in G$,  there exists the  unit morphism $1_x:x\to x$ in ${\bf G}$. These unit morphisms  together define  the unit map
$u:G\to {\bf G}$ by $u(x)=1_x$. Finally, for every morphism $g:x\to y$ in $ {\bf G}$,  there exists
the inverse morphism $g^{-1}:y\to x$.
These inverses together define the inverse
map $i:{\bf G}\to {\bf G}$ by $i(g)=g^{-1}$. The orbit space of the groupoid $\mathfrak{G}$,
$$\abs{\mathfrak{G}}=G/\sim $$
is  the quotient of the objects by the equivalence relation $\sim$
defined by $x\sim y$ if and only if there exists a morphism $g:x\to
y$ between the two objects. The equivalence class $\{y\in G\vert \,
y\sim x\}$ will be denoted by $|x|$. For fixed $x\in G$,  we
denote by   $G_x$
 its {\bf isotropy group} defined by
$$G_x=\{ \text{morphisms\, } g:x\to x\}.$$
For the  sake of notational economy we shall denote in the following a groupoid as well as its object  set by the same letter $G$ and its morphism set by the bold letter ${\bf G}.$

Ep-groupoids, as defined next, can be viewed as M-polyfold versions of  \'etale and proper Lie-groupoids discussed e.g. in \cite{Mj} and \cite{MM}.

\begin{defn}
An  {\bf ep-groupoid}  $X$ is a groupoid $X$ together with M-polyfold
structures on the object set $X$ as well as on the morphism set ${\bf X}$
so that all the structure maps  $(s, t, m, u, i)$ are sc-smooth maps and the following
holds true.
\begin{itemize}
\item {\em ({\bf \'etale})} The source and target maps
$s$ and $t$ are surjective local sc-diffeomorphisms.
\item {\em ({\bf proper})} For every point $x\in X$,    there exists an
open neighborhood $V(x)$ so that the map
$t:s^{-1}(\overline{V(x)})\rightarrow X$ is a proper mapping.
\end{itemize}
\end{defn}

We  point out that if  $X$  is a groupoid  equipped  with
M-polyfold structures on the object set $X$ as well as on the
morphism set ${\bf X},$ and  $X$  is \'etale,  then  the
fibered product ${\bf X}{{_s}\times_t}{\bf X}$ has a natural
M-polyfold structure so that the multiplication map $m$ is defined
on an  M-polyfold.  Hence  it makes sense to talk about its
sc-smoothness. For a proof we refer to \cite{HWZ3}.

In order to have a smooth partition of unity available, we shall assume in the
following that the sc-structure on the ep-groupoid $X$ is based on
separable sc-Hilbert spaces.

In an ep-groupoid every morphism $g:x\to y$ can be extended to
a unique local diffeomorphism $t\circ s^{-1}$ satisfying $s(g)=x$
and $t(g)=y$. The properness assumption implies that  the isotropy groups
 $G_x$ are finite groups. The degeneracy index
 $d:X\to \N_0$ is defined on the M-polyfold $X$
 of objects as well as on the M-polyfold ${\bf X}$ of morphisms.
 An M-polyfold $X$ is equipped with a filtration
$$X=X_0\supset X_1\supset X_2\supset \cdots \supset X_{\infty}:=
\bigcap_{k\geq 0}X_k.$$
The set  $X_{\infty}$ is dense in every space $X_k$.  If $x$ belongs to
the M-polyfold  $X$, we denote by $\text{ml}(x)\in
\N_0\cup\{\infty\}$ the largest integer $m$ or $\infty$ so that
$x\in X_m$. Since in an ep-groupoid source and the target maps are
local sc-diffeomorphisms, and therefore preserve by definition the
levels, we conclude that
$$\text{ml }(x)=\text{ml}(y)=\text{ml}(g)$$
for every morphism $g:x\to y$.
Consequently, the above filtration on the  M-polyfold $X$  induces the filtration
$$\abs{X}=\abs{X_0}\supset  \abs{X_1}\supset \abs{X_2}\supset \cdots \supset \abs{X_{\infty}}:=\bigcap_{k\geq 0}\abs{X_k}$$
on the orbit space $\abs{X}=X/\sim.$
It follows
 from the corner recognition Proposition 3.14 in \cite{HWZ1}  that also the degeneracy index  $d:X\to \N_0$ satisfies
$$d(x)=d(y)=d(g)$$
for every morphism  $g:x\to y$.  Hence the filtration of the boundary $\partial X$  descents to the  filtration of the  orbits space $\abs{\partial X}.$

The  next  result from \cite{HWZ3} describes the local structure of the morphism set of an ep-groupoid in the neighborhood of an isotropy group.

\begin{prop}\label{prop1.2}
Let $x$  be an object of an ep-groupoid $X$.  Then there exist an open neighborhood
$U\subset X$ of $x$,  a group homomorphism
$$
\varphi:G_{x}\to  \text{Diff}_{\ssc}(U), \quad g\mapsto
\varphi_g,
$$
of the isotropy group into the group of sc-diffeomorphims
of $U$,  and an sc-smooth map
$$
\Gamma:G_x\times U\to {\bf X}
$$
having the following properties.
\begin{itemize}
\item[$\bullet$] $\Gamma(g,x)=g$.
\item[$\bullet$] $s(\Gamma(g,y))=y$ and $t(\Gamma(g,y))=\varphi_g(y)$.
\item[$\bullet$] If $h:y\rightarrow z$  is a morphism between two points in  $U$,  then there exists a
unique element $g\in G_x$ satisfying  $\Gamma(g,y)=h$.
\end{itemize}
\end{prop}
The group homomorphism $\varphi:G_x\to \text{Diff}_{\ssc}(U)$ is
called a {\bf natural representation of the isotropy group $G_x$} of
the element $x\in X$. The diffeomorphism $\varphi_g$ is given by
$t\circ s^{-1}$ where $s(g)=t(g)=x$. We see that every morphism
between points in $U$ belongs to the image of the map $\Gamma$ and
so has an extension to precisely one of the finitely many
diffeomorphisms $\varphi_g$ of $U$, where $g\in G_x$.

The next definition of a  branched ep-subgroupoid
generalizes ideas  from  \cite{CRS} where quotients of
manifolds by global group actions are considered.
We shall view the nonnegative rational numbers, denoted by $\Q^+=\Q\cap [0,\infty )$,
as the objects in a category having only the identities
as morphisms. We would like to mention  that branched
ep-subgroupoids will show up as solution sets of polyfold Fredholm
sections.

\begin{defn}\label{def1}  A {\bf  branched ep-subgroupoid of $X$}  is  a
functor
 $$\Theta:X\rightarrow {\mathbb Q}^+$$  having the following properties.
\begin{itemize}
\item[(1)] The support  of $\Theta$, defined by $\supp \Theta=\{x\in X\vert \, \Theta(x)>0\}$,  is
contained in $X_\infty$.
\item[(2)] Every  point $x\in \supp \Theta$ is contained in an open
neighborhood $U(x)=U\subset X$ such that
$$\supp \Theta \cap U= \bigcup_{i \in I}M_i,$$
where $I$ is a finite index set and where the sets $M_i$  are finite dimensional submanifolds of $X$ (Definition 4.20 in \cite{HWZ2}, see also Appendix \ref{fdimsubm}) all having the same dimension, and all in good position to the boundary $\partial X$. The submanifolds $M_i$ are called {\bf local branches} in $U$.
\item[(3)] There exist  positive rational numbers $\sigma_i$, $i\in I$, (called {\bf weights}) such that if $y\in \supp \Theta \cap U$, then
$$
\Theta(y)=\sum_{\{i \in I \vert   y\in M_i\}} \sigma_i.
$$
\item[(4)] The inclusion maps $M_i\to U$ are proper.
\item[(5)] There is a natural representation of the isotropy group $G_x$ acting by sc-diffeomorphisms on $ U$.
\end{itemize}
\end{defn}

\begin{figure}[htbp]
\mbox{}\\[2ex]
\centerline{\relabelbox
\epsfxsize 3.1truein \epsfbox{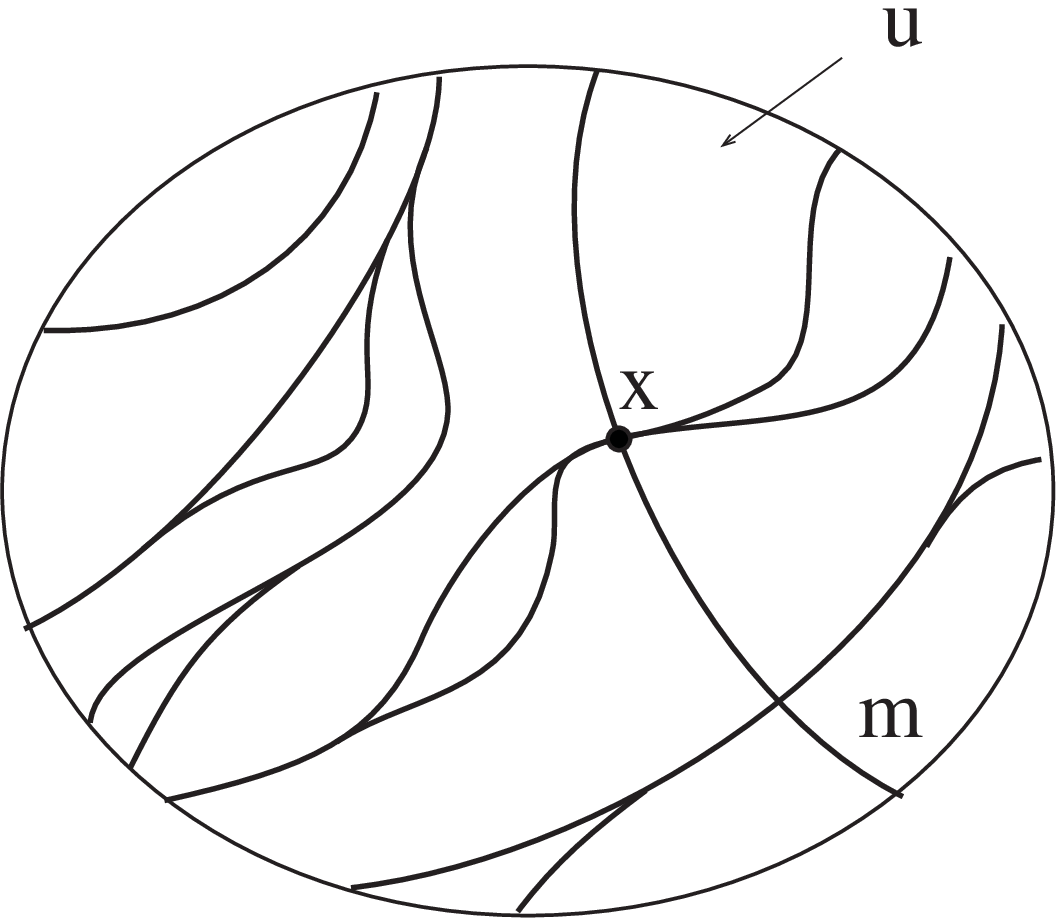}
\relabel {m}{$M_j$}
\relabel {u}{$U$}
\relabel {x}{$x$}
\endrelabelbox}
\caption{}\label{Fig1}
\mbox{}\\[1ex]
\end{figure}

The branches  $(M_i)_{i\in I}$ together with the weights
$(\sigma_i)_{i\in I}$ constitute  a {\bf local branching structure}
of $X$ in $U$.
\begin{rem}
Given a branched ep-subgroupoid $\Theta$ on the ep-groupoid $X$  we obtain
induced ep-subgroupoids $\Theta^i:X^i\rightarrow {\mathbb Q}^+$ for $i\geq 1$. Observe that the supports of $\Theta=\Theta^0$ and $\Theta^i$ coincide. The reader should notice that all upcoming constructions carried out for $\Theta$ will lead to the same result if carried out for $\Theta^i$.
\end{rem}

\begin{defn}
A branched ep-subgroupoid is called  {\bf $\mathbf{n}$-dimensional} if all
its local branches are n-dimensional submanifolds of $ M.$
\end{defn}

From the assumption that $\Theta:X\to \Q^+$ is a functor it follows that
$$\Theta (x)=\Theta (y)$$
if there is a morphism $g:x\to y$. We shall abbreviate the support of $\Theta$ by
$$S_{\Theta}=\supp \Theta =\{x\in X\vert \, \Theta (x)>0\}.$$
If $x\in S_{\Theta}$,  we can view the tangent space $T_xM_i$ of a local branch as a linear subspace of the tangent space $T_xX$. The tangent set $T_xS_{\Theta}$ is the union of the tangent spaces $T_xM_i$ which are called the local branches of the tangent set $T_xS_{\Theta}$. As we shall see later on, the tangent set $T_xS_{\Theta}$ is independent of the local branching structure. The union
$$TS_{\Theta}=\bigcup_{x\in S_{\Theta}}\{x\}\times T_xS_{\Theta}$$
can be viewed as a subset of the tangent bundle $TX$. It is
called the {\bf tangent set of the support $\mathbf{S}_{\mathbf{\Theta}}$}.

If  $g:x\rightarrow y$ is a morphism
between points $x,y\in X_1$, then $g$ belongs to ${\bf X}_1$  and extends to an sc-diffeomorphism  $t\circ s^{-1}:O(x)\rightarrow O(y)$. We define $Tg:T_xX\rightarrow
T_yX$ to be the linearization of $t\circ s^{-1}$ at $x$. If
$x\in \supp \Theta$, then  the  tangent map $Tg$  takes  a (linear) local branch of
$T_xS_\Theta$ to a local branch of $T_yS_\Theta$.

Next we define the notion of an orientation
of a branched ep-subgrou\-poid.

\begin{defn}
Let $\Theta:X\to \Q^+$ be a branched ep-subgroupoid on  the ep-groupoid $X$.
An {\bf orientation  for $\Theta$},   denoted by $\mathfrak{o}$,   consists of an
orientation for every local branch of the tangent set
$T_xS_{\Theta}$ at every point  $x\in S_{\Theta}$ so that the
following compatibility conditions  are satisfied.
\begin{itemize}
\item[(1)]
At  every point $x$,  there
exists a local branching structure $(M_i)_{i\in I}$ where the finite dimensional submanifolds  $M_i$ of $X$
can be oriented in such a way  that the orientations of $T_xM_i$ induced from $M_i$ agree with the given ones for the local branches of the tangent set $T_xS_{\Theta}$.
\item[(2)] If  $\varphi:x\rightarrow y$ is a morphism between two points in $S_{\Theta}$, then the tangent map
 $T\varphi:T_xS_\Theta\rightarrow T_yS_\Theta$ maps every oriented branch to
 an oriented branch preserving the
orientation.
\end{itemize}
\end{defn}

After these recollections and definitions we shall formulate the main results.

\subsection{Main Results for Ep-Groupoids}
In this section we state the main results in the ep-groupoid context.
We consider a  branched ep-subgroupoid $\Theta:X\to \Q^+$ of dimension $n$ on the ep-groupoid $X$ and assume that the orbit space of $\Theta$, $S=\abs{\supp \Theta}$,  is compact. We abbreviate the orbit space of the boundary  by
$$\partial S=\abs{\supp \Theta
\cap \partial X}=\{ |x|\in S\vert \, x\in \supp \Theta \cap \partial X\}.
$$
The orbit space $\partial
S$, in general, needs not to be the boundary of $S$ in the sense of
topology. Locally the sets $S$ and $\partial S$ can be
represented as quotients by smooth group actions on finite unions of
smooth manifolds. This structure suffices to generalize the familiar
notion of the Lebesgue $\sigma$-algebra of subsets of smooth
manifolds as follows.
\begin{thm}[{\bf Canonical $\sigma$-Algebra}]\label{th0}
 The compact topological spaces $S$ and $\partial S$ possess canonical $\sigma$-algebras ${\mathcal
L}(S)$ and ${\mathcal L}(\partial S)$ of subsets containing the
Borel $\sigma$-algebras of $S$ and $\partial S$, respectively.
\end{thm}
The sets belonging to ${\mathcal L}(S)$ and ${\mathcal L}(\partial S)$
are  called measurable. The precise definition of the $\sigma$-algebras will
be given in the later construction. The canonical $\sigma$-algebras for $\Theta$
and  $\Theta^i$ are identical.

From
\cite{HWZ1} and \cite{HWZ3} we know that if $X$ is an ep-groupoid,
then the tangent space  $TX$ is again an ep-groupoid. Recall that a
morphism $\varphi\in {\bf X}_{1}$ possesses a well-defined tangent
morphism $T\varphi\in {\bf TX}$. In  the following, we  denote by
$\oplus_kTX\rightarrow X^1$  the Whitney sum of $k$-many copies of
$TX$.

At this point we  would like to make some comments about our  notation.
For us the tangent bundle of $X$ is $TX\rightarrow X^1$,  that is, it is  only defined for the base points in $X^1$. An   {\bf sc-vector field} on $X$  is an  sc-smooth section $A$ of  the tangent bundle $TX\rightarrow X^1$ and hence it is  defined on $X^1$. Similarly, an  sc-differential form  on $X$ which we will define next,  is  only defined over the base points in $X^1$. The definition of a vector field and the following  definition of an sc-differential form
are justified since the construction of $TX$, though only defined over $X^1$, requires the knowledge of $X$.
\begin{defn}
An {\bf $\boldsymbol{\ssc}$-differential ${\boldsymbol{k}}$-form on the $M$-polyfold $X$}  is
an sc-smooth map $\omega:\oplus_k TX \to \R$ which is
linear in each argument separately, and skew symmetric.   In addition,  we require that
$$(T\varphi)^\ast\omega_y=\omega_x$$
when  $\varphi:x\rightarrow y$ is in ${\bf X}_1.$
\end{defn}

If $\omega$ is an  sc-differential form on $X$,  we may also  view it  as an  sc-differential form on $X^i$.  Denote by $\Omega^\ast(X^i)$ the graded commutative algebra of sc-differential forms on $X^i$. Then we have the inclusion map
$$
\Omega^\ast(X^i)\rightarrow\Omega^\ast(X^{i+1}).
$$
which  is injective since $X_{i+1}$ is dense in $X_i$ and the forms are sc-smooth. Hence we have a directed system whose  direct limit is denoted by
$\Omega^\ast_\infty(X)$. An element $\omega$ of  degree $k$ in $\Omega^\ast_\infty(X)$ is a skew-symmetric map $\oplus_k(TX)_\infty\rightarrow {\mathbb R}$ such that it has an  sc-smooth extension to an  sc-smooth k-form
$\oplus_kTX^i\rightarrow {\mathbb R}$ for some $i$. We shall refer to an element of $\Omega^k_\infty(X)$ as an  sc-smooth differential form on $X_\infty$. We note, however, that it is part of the structure that the $k$-form is defined and sc-smooth on some $X^i$.

Next we associate with an sc-differential $k$-form $\omega$ its exterior differential
$d\omega$ which is a $(k+1)$-form on the ep-groupoid $X^1$. The compatibility with morphisms will be essentially an automatic consequence of the definition. Hence it suffices for the moment to consider the case that $X$ is an M-polyfold.
Let $A_0, \ldots , A_k$ be $k+1$ many sc-smooth vector fields on $X$. As discussed in Appendix \ref{poly}, the Lie bracket $[A_i, A_j]$ is a well-defined vector field on $X^1$. We define
$d\omega$ on $X^1$, using the familiar  formula,   by
\begin{equation*}
\begin{split}
d\omega (A_0, \ldots ,A_k)&=\sum_{i=0}^k(-1)^iD(\omega (A_0, \ldots , \wh{A}_i, \ldots , A_k))\cdot A_i\\
&+\sum_{i<j}(-1)^{(i+j)}\omega ([A_i, A_j], A_0, \ldots , \wh{A}_i,\ldots ,\wh{A}_j, \ldots ,A_k).
\end{split}
\end{equation*}
The right-hand side of the formula above only makes sense at the base points $x\in X_2$. This  explains why  $d\omega$ is a $(k+1)$-form on $X^1$.
By the previous discussion the differential $d$ defines a map
$$
d:\Omega^k(X^i)\rightarrow \Omega^{k+1}(X^{i+1})
$$
and consequently induces a map
$$
d:\Omega^\ast_\infty(X)\rightarrow \Omega^{\ast+1}_\infty(X)
$$
having the usual property $d^2=0$. Then $(\Omega^\ast_\infty(X),d)$ is a  graded differential algebra which we shall call the de Rham complex.

If $\varphi:M\to X$ is an sc-smooth map from a finite-dimensional manifold $M$ into an M-polyfold $X$, then it induces an  algebra homomorphism
 $$
 \varphi^\ast:\Omega^\ast_\infty(X)\rightarrow \Omega^\ast(M)_\infty
 $$
 satisfying
$$d (\varphi^*\omega )=\varphi^*d\omega.$$

To formulate the  next theorem we recall the natural representation $\varphi:G_x\to \text{Diff}_{\ssc}(U)$ of the isotropy group $G_x=G$. This group can  contain a subgroup  acting trivially on $U$. Such a subgroup is a normal subgroup of $G$ called the ineffective part of $G$ and denoted by $G_0$. The  effective part of $G$ is the quotient group
$$G_{\text{e}}:=G/G_0.$$
We  denote  the order of $G_{\text{e}}$ by $\sharp G_{\text{e}}$.

\begin{thm}[{\bf Canonical Measures}]\label{th1}
Let $X$ be an ep-groupoid and assume that $\Theta:X\to \Q^+$ is an oriented branched
ep-subgroupoid of dimension $n$ whose orbit space $S=\abs{\supp \Theta}$  is compact and equipped with the weight function
$\vartheta:S\to \Q^+$ defined by
$$
\vartheta (|x|):=\Theta (x),\quad |x|\in S.
$$
Then there exists a map
$$
\Phi_{(S,\vartheta)}:\Omega^n_\infty(X)\rightarrow {\mathcal M}(S,{\mathcal
L}(S)),\quad \omega \mapsto    \mu_\omega^{(S,\theta)}
$$
which  associates to every  sc-differential $n$-form $\omega$ on $X_\infty$ a signed finite measure
$$\mu_{\omega}^{(S, \vartheta )}\equiv \mu_{\omega}$$
on  the canonical measure space $(S,{\mathcal L}(S))$. This map is uniquely characterized by the
following properties.
\begin{itemize}
\item[(1)] The map $\Phi_{(S,\theta)}$ is linear.
\item[(2)] If $\alpha=f\tau$ where $f\in\Omega^0_\infty(X)$ and $\tau\in\Omega^{n}_\infty(X)$, then
$$
\mu_\alpha(K)=\int_K
fd\mu_\tau
$$
for  every set $K\subset S$ in the $\sigma$-algebra ${\mathcal L}(S)$.
\item[(3)] Given a  point $x\in \supp \Theta$ and an oriented branching structure
${(M_i)}_{i\in I}$  with associated weights $(\sigma_i)_{i\in I}$ on the open neighborhood  $U$ of $x$ according to Definition \ref{def1}, then  for every set  $K\in {\mathcal L}(S)$ contained in a compact subset of $\abs{\supp \Theta \cap U}$,  the $\mu_{\omega}$-measure of $K$ is given by the formula
$$
\mu_\omega (K)=\frac{1}{\sharp G_{e}}\sum_{i\in I}
\sigma_i\int_{K_i}\omega\vert M_i
$$
where $K_i\subset M_i$ is the preimage of $K$ under  the projection map
$M_i\to \abs{\supp \Theta \cap U}$ defined by $x\rightarrow |x|$.
\end{itemize}
\end{thm}

In the theorem above we denoted  by $\int_{K_i}\omega\vert M_i$ the signed measure of the set
$K_i$ with respect to the Lebesgue signed measure associated with the smooth $n$-form $\omega\vert M_i$ on the finite dimensional manifold $M_i$. It is given by
$$\int_{K_i}\omega\vert M_i=\lim_k\int_{U_k}j^*\omega$$
where $j:M_i\to X$  is the inclusion map and $(U_k)$ is  a decreasing sequence of open neighborhoods of $K_i$ in $M_i$ satisfying $\bigcap_k U_k=K_i$.

The analogous result holds true for the orbit space $\partial S=\{|x|\in S\vert \,  x\in \supp \Theta\cap \partial X\}$ of the boundary.

\begin{thm}[ {\bf Canonical Boundary  Measures}] Under the same assumptions
as in Theorem \ref{th1} there exists a map
$$
\Phi_{(\partial S,\vartheta)}:\Omega^{n-1}_\infty(X)\rightarrow {\mathcal
M}(\partial S,{\mathcal L}(\partial S)), \quad \tau\mapsto
\mu_\tau^{(\partial S,\vartheta)}
$$
which assigns to every sc-differential $(n-1)$-form $\tau$ on $X_\infty$ a signed finite measure
$$\mu_{\tau}^{(\partial S, \vartheta)}\equiv \mu_{\tau}$$
on  the canonical measure  space
$(\partial  S, {\mathcal L} (\partial S))$.
 This map is uniquely characterized by the
following properties.
\begin{itemize}
\item[(1)] The map $\Phi_{(\partial S,\vartheta )} $ is linear.
\item[(2)] If $\alpha=f\tau$ where  $f\in\Omega^0_\infty(X)$ and $\tau\in\Omega^{n-1}_\infty(X)$, then every $K\in {\mathcal L}(\partial S)$
has the $\mu_{\alpha}$-measure
$$\mu_\alpha(K)=\int_K
fd\mu_\tau.$$
\item[(3)] Given a  point $x\in \supp \Theta \cap \partial X$  and an oriented branching structure  ${(M_i)}_{i\in I}$ with weights $(\sigma_i)_{i\in I}$ on the open neighborhood $U\subset X$ of $x$, then the measure of  $K\in {\mathcal L}(\partial S)$  contained in a compact subset of $\abs{\supp \Theta \cap U\cap \partial X}$ is given by the formula
$$
\mu_\tau(K)=\frac{1}{\sharp G_{e}}\sum_{i\in I}
\sigma_i\int_{K_i}\tau \vert \partial M_i
$$
where $K_i\subset  \partial M_i$ is the preimage of $K$ under the projection map  $\partial M_i\to \abs{\supp \Theta \cap U\cap \partial X}$ defined by $x\mapsto |x|$.
\end{itemize}
\end{thm}
Finally,  the following   version of Stokes' theorem holds.

\begin{thm}[{\bf Stokes Theorem}]\label{thmst0}
Let $X$ be   an ep-groupoid and let  $\Theta:X\to \Q^+$ be an oriented $n$-dimensional
branched ep-subgroupoid of $X$  whose orbit space  $S=\abs{\supp \Theta}$ is compact. Then,
for every  sc-differential $(n-1)$-form $\omega$ on $X_\infty$,
$$
\mu^{(S,\vartheta)}_{d\omega}(S)=\mu^{(\partial S,\vartheta)}_{\omega}(\partial
S),
$$
or alternatively,
$$\int_{(S, \vartheta)}d\omega =\int_{(\partial S, \vartheta)}\omega.$$
\end{thm}

To demonstrate the  compatibility of the construction with
equivalences between ep-groupoids,  we first introduce the notion of  an equivalence and refer the reader to  \cite{HWZ3} for more details.

\begin{defn}[{\bf Equivalence}]
An sc-smooth functor $F:X\to Y$ between two ep-groupoids is called an {\bf equivalence}  if it possesses the following properties.
\begin{itemize}
\item[(1)] $F$ is a local sc-diffeomorphism on objects as well as on
morphisms.
\item[(2)] The induced map $\abs{F}:\abs{X}\to \abs{Y}$ between the
orbit spaces is an $\ssc$-homeomorphism.
\item[(3)] For every $x\in X$,  the map $F$ induces a bijection
$G_x\rightarrow G_{F(x)} $ between the isotropy groups.
\end{itemize}
\end{defn}

Because  an equivalence $F:X\to Y$ is a local sc-diffeomorphism, we conclude for the degeneracy indices of the ep-groupoids $X$ and $Y$ that
$$d_X(x)=d_Y(F(x))$$
for every $x\in X$.
\begin{thm}[{\bf Equivalences}]\label{th4}
Assume that  $F:X\to  Y$ is an equivalence between the
ep-groupoids $X$ and $Y$.  Assume that  $\Theta:Y\rightarrow {\mathbb Q}^+$ is  an
oriented $n$-dimensional branched ep-subgroupoid of $Y$  whose orbit
space $S=\abs{\supp \Theta}$ is compact and equipped with the weight
function $\vartheta:S\to \Q^+$ defined by $\vartheta (|y|)=\Theta
(y)$ for $|y|\in S$. Define the $n$-dimensional branched
ep-subgroupoid  on $X$ by $\Theta':=\Theta\circ F:X\to \Q^+$ and denote by $S'$ and $\vartheta'$  the associated orbit space   and the  weight function on $S'$. Moreover, assume that
$\Theta'$ is equipped with the induced orientation. Then, for every
sc-differential $n$-form $\omega$ on $Y_\infty$,
$$
\mu_{\omega}^{(S,\vartheta)}\circ \abs{F} = \mu_{\omega'}^{(S',\vartheta')}
$$
where the $n$-form $\omega'$ on $X_\infty$ is the pull back form $\omega'=F^*\omega$.
Similarly,
$$
\mu_{\tau}^{(\partial S,\vartheta)}\circ \abs{F}=
\mu_{\tau'}^{(\partial S',\vartheta')}
$$
for every $(n-1)$-form $\tau$ on $Y_\infty$.
\end{thm}

\subsection{Main Results for Polyfolds}
The definition of an equivalence and Theorem \ref{th4}  allow us  to
rephrase the previous results in the  polyfold set-up.
We begin by defining the concept of a polyfold. This  concept is discussed in much detail in \cite{HWZ3}. We  consider  a second countable paracompact space $Z$. A polyfold structure on $Z$
is a pair $(X,\beta)$ in which $X$ is an ep-groupoid and
$\beta:\abs{X}\rightarrow Z$ a
homeomorphism from the orbit space of $X$ onto $Z$.
Two such polyfold structures $(X,\beta)$ and
$(X',\beta')$ on $Z$  are called  equivalent,
$$(X, \beta )\sim (X', \beta'),$$
if there exists a  third polyfold structure $(X'', \beta'')$ on $Z$ and two equivalences
$$
X\xleftarrow{F} X''\xrightarrow{F'} X'
$$
satisfying
$$\beta''=\beta\circ \abs{F} = \beta' \circ{F'}.$$ This  way we obtain indeed an equivalence relation between polyfold structures on $Z$ as is demonstrated in \cite{HWZ3}, and we can introduce the following definition.

\begin{defn}
A {\bf polyfold} is  a second countable paracompact space
$Z$ equipped with an equivalence class of polyfold structures.
\end{defn}

Consider a   polyfold  $Z$  and  a polyfold structure  $(X, \beta)$  on $Z$. The orbit space $\abs{X}$ of the ep-groupoid $X$ is equipped with a filtration
$\abs{X_0}=\abs{X}\supset \abs{X_1}\supset \cdots \supset  \abs{X_{\infty}}:=
\abs{\bigcap_{i\geq 0}X_i}$
where   $\abs{X_{\infty}}$ is dense in every $\abs{X_i}$. This  filtration induce,  via the homeomorphism $\beta:\abs{X}\to Z$,  the  filtration
$$Z_0
=Z\supset  Z_1\supset \cdots  \supset Z_{\infty}:
=\bigcap_{i\geq 0}Z_i$$
on the polyfold $Z$ in which the space $Z_{\infty}$ is dense in every $Z_i$.  Any  other   equivalent polyfold structure on $Z$ induces the same filtration. Every space $Z_i$ carries a polyfold structure which we denote by $Z^i$. 

An sc-differential form on the topological space $Z$ or $Z^i$ is defined
via the overhead of the postulated polyfold structures. We define an
sc-differential form $\tau$ on the polyfold $Z$ as an equivalence
class of triples  $(X, \beta , \omega)$ in which the pair $(X,
\beta)$ is a polyfold structure on $Z$ and $\omega$ an
sc-differential form on $X$. Two such  triples $(X, \beta , \omega)$
and $(X', \beta' , \omega')$ are called equivalent, if there is a
third ep-groupoid $X''$ and two equivalences $X\xleftarrow{F}
X''\xrightarrow{F'} X'$ satisfying $\beta\circ \abs{F} = \beta'\circ
\abs{F'}$ and, in addition,
$$
F^*\omega =(F')^*\omega'
$$ for the pull back forms on $X''$.
We shall abbreviate an equivalence class by $\tau =[\omega]$ where $(X, \beta ,\omega)$ is a representative of the class and call $\tau$ an sc-differential form on the polyfold $Z$. Again we have a directed system obtained from the inclusion $Z^{i+1}\rightarrow Z^i$ which allows us to define, analogously to the ep-groupoid-case, the notion of a differential form on $Z_\infty$. We shall  still use the symbol $[\omega]$ for such forms.  It is important to keep in mind that a representative $\omega$ of $[\omega]$ is an  sc-differential form on $X_\infty$ where $(X,\beta)$ is a polyfold structure on  $Z$.

The exterior derivative of an sc-differential form $[\omega ]$ is defined by
$$
d[\omega] =[d\omega].
$$

A {\bf branched suborbifold $S$ of a polyfold $Z$}
 is a subset $S\subset Z$ equipped with a weight
 function $w:S\to \Q^+\cap (0, \infty )$ together
 with an equivalence class of triples
 $(X, \beta ,\Theta)$ in which the pair $(X, \beta)$
 is a polyfold structure on $Z$ and $\Theta:X\to \Q^+$
 is an ep-subgroupoid  of $X$ satisfying $S=\beta (\abs{\supp \Theta })$
 and $w (\beta |x|)=\Theta (x)$ for $x\in \supp\Theta$.
Two  triples  $(X, \beta , \Theta )$, $(X', \beta', \Theta')$ are said to be
equivalent,
$$(X, \beta , \Theta )\sim (X', \beta', \Theta')$$
if there exists  a third triple $(X'', \beta'', \Theta'')$ and
two equivalences
$X\xleftarrow{F} X''\xrightarrow{F'} X'$
satisfying $\beta''=\beta \circ \abs{F}=\beta'\circ \abs{F'}$
for the induced maps on the orbit spaces and, in addition,
$$\Theta''=\Theta'\circ F'=\Theta \circ F.$$

The branched suborbifold $S\subset Z$ is called $n$-dimensional, resp. compact, resp. oriented, if the representative $(X, \beta ,\Theta)$ of the ep-groupoid is $n$-dimensional, resp. $\abs{\supp \Theta}$ is compact, resp. $\Theta$ is oriented. In this latter case,  the equivalences $F$ and $F'$ are required to preserve the orientations.

Similarly,  we define  the ``boundary'' set $\partial S$ of a
branched suborbifold $S$ by setting
$\partial S=\beta (\abs{\supp \Theta \cap \partial X})$
 for a representative $(X, \beta , \Theta)$ of the equivalence class.
 From the  previous results it follows
that for a compact and oriented
  branched suborbifold  $S\subset Z$ of a polyfold $Z$ there is a canonical
$\sigma$-algebra $(S, {\mathcal L}(S))$ of measurable subsets and a
well defined integration theory for which Stokes' theorem holds.

\begin{thm}\label{thmst1}
Let $Z$ be a polyfold and $S\subset Z$ be an oriented compact branched suborbifold
defined by the equivalence class $[(X, \beta ,\Theta)]$ and equipped with the weight function $w:S\to \Q^+\cap (0,\infty )$.
For an sc-differential $n$-form $\tau$ on $Z_\infty$ and $K\in {\mathcal L}(S)$,  define
$$
\int_{(K,w)}\tau:=\int_{\beta^{-1}(K)}
d\mu_\omega^{(\beta^{-1}(S),\vartheta)}=\mu_\omega^{(\beta^{-1}(S),\vartheta)}(\beta^{-1}(K)),
$$
where  the equivalence class $\tau$ is represented by the triple $(X, \beta ,\omega)$ and the weight function $\vartheta$  on
$\beta^{-1}(S)=\abs{\supp \Theta}$ is defined by $\vartheta (|x|)=\Theta (x).$
Then    the integral  $\int_{(K,w)}\tau$ is
well defined. Moreover,   if $\tau$ is an sc-differential $(n-1)$-form on $Z_\infty$, then
$$
\int_{(\partial S,w)}\tau =\int_{(S,w)} d\tau.
$$
\end{thm}

In the statement of Theorem \ref{thmst1} we associated with a given differential form $\tau$ and a given oriented compact branched suborbifold $S$ equipped with  the weight function $w$,
a  measure $\mu_\tau^{(S,w)}$ on $(S,{\mathcal
L}(S))$ which incorporates the weight function $w$.
Alternatively,
one could associate a measure on $(S, {\mathcal L}(S))$ which does not
take the weighting into account. Then the formulae occurring would
be different. This alternative measure would be given by
$\frac{1}{w}d\mu_\tau^{(S,w)}$. For the applications to the Fredholm
theory  the following point of view is useful. The solution sets
are of the form $|\supp \Theta|$ which gives a topological space
$(S,w)$ equipped with a weighting and  the differential form
descends to a differential form on $S$ and is integrated taking into
account that points in $S$ only account as prescribed by their
weight. From a practical point of view, i.e.,  computations, the
alternative view point that the differential form gives a signed
measure on the un-weighted $S$ is useful. We have used both views
and the formalism is given in such a way that any view point can be
emphasized. If we write $\int_{(K,w)}\tau$,  the differential form
$\tau$ is integrated over the weighted space $(K,w)$, $K\subset
(S,w)$.  If  we write $\mu_{\tau}^{(S,w)}(K)$, then   we take the
measure of the un-weighted $K$ and the weight is already encoded in
the measure. Of course,
$$
\int_{(K,w)}\tau =\mu_{\tau}^{(S,w)}(K).
$$

\subsection{Applications to Polyfold Fredholm Sections}
In this subsection we shall sketch an application of the branched integration theory, referring
for details to \cite{HWZ3}. We assume that all the spaces are build on splicing cores based on separable Hilbert spaces, so that  sc-smooth partitions of unity are
available, see Appendix \ref{scpartition}.

We consider a strong polyfold bundle $p:W\to Z$ as introduced in Definition 3.15 of  \cite{HWZ3} and a proper Fredholm section $f$ of the bundle $p$ as defined in Definition 4.1 of \cite{HWZ3}. The term ``proper" requires the solution set $\{z\in Z\vert \, f(z)=0\}$ to be a compact subset of $Z$. By definition of a polyfold bundle, there exist as overhead a strong bundle $P:E\to X$ over the ep-groupoid $X$ together with two
homeomorphisms $\Gamma:\abs{E}\to W$ and $\gamma:\abs{X}\to Z$ defined on the orbit spaces and satisfying $p\circ \Gamma =\gamma \circ \abs{P}$. Moreover, again by  definition of $f$, there exists a Fredholm section $F$ of the strong bundle $P:E\to X$ representing the section $f$ in the sense that
 $f\circ \gamma =\Gamma \circ \abs{F}$. We recall from \cite{HWZ3} the concept of $\ssc^{+}$-multisections.

\begin{defn}\label{muli}
Let   $P:E\to X$  be a strong bundle over the ep-groupoid $X$. Then  an
 $\mathbf{sc}^{+}$-{\bf multisection}   of $P$  is a functor
$$\Lambda:E\to \Q^+$$
such that the following local representation, called {\bf local section structure},   holds true. For every object $x_0\in X$,   there
exist  an open neighborhood $U(x_0)\subset X$, finitely many
 $\text{sc}^+$-sections $s_1,\ldots ,s_k:U\to E$, called {\bf local sections}, and associated positive rational numbers
$\sigma_1. \ldots ,\sigma_k$, called {\bf weights}, satisfying
\begin{equation*}
\sum_{j=1}^k\sigma_j=1 \quad   \text{and}\quad \Lambda (e)=\sum_{\{j\vert\,
s_j(P(x))=e\}}\sigma_j
\end{equation*}
for all $e\in E\vert U$. By convention, the sum over the
empty set is equal to $0$.
\end{defn}

The functoriality of $\Lambda$ implies that $\Lambda (e)=\Lambda (e')$ if there exists a morphism $e\to e'$ in ${\bf E}$. Hence $\Lambda$ induces a  map $\abs{\Lambda}:\abs{E}\to \Q^+$.

An $\ssc^+$-multisection of the polyfold bundle is a map $\lambda:W\to \Q^+$ for which there exists an
$\ssc^+$-multisection $\Lambda$ of the overhead $P:E\to X$ satisfying $\lambda (w)=\Lambda (e)$ if $w=\Gamma ([e])$.

It has been proved  in \cite{HWZ3} that there are many
$\ssc^+$-multisections $\lambda$ having their supports in preassigned {saturated open sets $U$ and the pair  $(f,\lambda)$ is in   good position to the boundary, so that  the representative $F$ of $f$ for the overhead model $P:E\to X$ of the bundle $p$ possesses an $\ssc^{+}$-mulitsection $\Lambda$, for which $\Theta=\lambda \circ F:X\to \Q^+$ is a branched ep-subgroupoid. Recall that a subset $A$ of $X$ is saturated provided it contains all points which are connected by a morphism to a point in $A$. The  support of $\Theta$ is the solution set
$$S (F, \Lambda )=\{x\in X\vert \, \Lambda (F(x))>0\}.$$
Of course, this is an example of a saturated set.
The solution set $S(f, \lambda )\subset Z$ of the pair $(f,\lambda )$ is  defined by
$$S(f, \lambda )=\{z\in Z\vert \, \lambda (f(z))>0\}.$$
The map $x\mapsto \gamma (|x|)$ from $S(F, \Lambda )$ to $S(f, \lambda )$ induces a homeomorphism
$\abs{S(F, \Lambda )}\to S(f, \lambda )$. The natural map $\lambda_f:S(f, \lambda ) \to \Q^+$ defined by
$$\lambda_f (z):=\lambda (f(z))$$
is called the weight function on $S(f, \lambda )$. If $f$ is oriented, then   $\supp \Theta$ has an induced orientation as shown in \cite{HWZ3}. Using the overheads we can integrate forms over the solution set $S(f,\lambda )\subset Z$.

Consider the differential algebra
$(\Omega_\infty^\ast(X),d)$ on the ep-groupoid $X$. Its cohomology,  denoted by $H^\ast_{dR}(X)$,  is called the de Rham cohomology of $X$. We have an inclusion $\partial X\rightarrow X$ and the restriction of this map to local faces is sc-smooth. Local faces have a natural M-polyfold structure and the same is true for the intersections  of local faces. Therefore,  it makes sense to talk about sc-differential forms on $\partial X$ and $\partial X^i$ as well as sc-differential forms on $\partial X_\infty$.
On  the differential algebra $\Omega^\ast_\infty(X)\oplus\Omega^{\ast-1}_\infty(\partial X)$
we define the differential $d$  by setting
$$
d(\omega,\tau)=(d\omega,j^\ast\omega-d\tau)
$$
where $j:\partial X\rightarrow X$ stands for   the inclusion map. Its  cohomology  is denoted  by $H^\ast_{dR}(X,\partial X)$.

\begin{lem}
Assume that $F:X\rightarrow Y$ is an equivalence between ep-groupoids. Then $F^\ast:\Omega_\infty^\ast(Y)\rightarrow\Omega^\ast_\infty(X)$ is an isomorphism of graded differential algebras. In particular,
$F$ induces an isomorphism $H^\ast_{dR}(Y)\rightarrow H^\ast_{dR}(X)$. Similarly,  $F^\ast:\Omega_\infty^\ast(Y,\partial Y)\rightarrow\Omega^\ast_\infty(X,\partial X)$ is an isomorphism of graded differential algebras and induces an isomorphism $H^\ast_{dR}(Y,\partial Y)\rightarrow H^\ast_{dR}(X,\partial X)$.
\end{lem}
\begin{proof}
It suffices  to prove  that $F^\ast:\Omega_\infty^\ast(Y)\rightarrow\Omega^\ast_\infty(X)$  is a bijection. We start with injectivity  and assume that
 $F^\ast\omega=0$ for some sc-differential form $\omega$ on $Y^i$. Since $F:X^i\to Y^i$ is a local sc-diffeomorphism, it follows  that $\omega$ vanishes on
$TF(TX^i)$.  In view of the definition of an equivalence, for every point in $Y^{i+1}$ there is a point in $F(X^{i+1})$  and a morphism connecting these two points. The same is true for points in $TY^i$ and $TF(TX^i)$.  These imply  that $\omega$ vanishes everywhere since $\omega$ is compatible with morphisms.  To prove surjectivity,   take  a form $\omega$ on $X^i$.  Then  the push-forward of the form $\omega$ by the equivalence $F$ defines a form $F_\ast\omega$ on $F(X^i)$. Since all  the points in $F(X^{i})$ are connected by morphisms  to points in $Y^{i}$ and $\omega$ is compatible with morphisms,  we can extend the form $F_\ast\omega$ to obtain a form $\tau$ on the whole $Y^i$.  Then $F^\ast \tau=\omega$.  For the relative case, we use that fact  that $F$ also induces equivalences between the local faces.
\end{proof}
The compatibility of sc-differential forms  with morphisms implies the following lemma.
\begin{lem}\label{lps}
Assume that $F,G:X\rightarrow Y$ are two equivalences between ep-groupoids and $\tau:F\rightarrow G$ is a natural transformation. Then the induced isomorphism $F^\ast$ and $G^\ast$ between $H^\ast_{dR}(Y)$  and $ H^\ast_{dR}(X)$ as well as $H^\ast_{dR}(Y,\partial Y)$ and $H^\ast_{dR}(X,\partial X)$ coincide.
\end{lem}
Lemma  \ref{lps} allows us to define the de Rham cohomology of a polyfold.
Consider a polyfold $Z$ and let  $(X,\beta)$ and $(X',\beta')$  be two  equivalent polyfold structures on $Z$. This means that there exists a third polyfold structure $(X'',\beta'')$ and two equivalences $F:X''\rightarrow X$ and $F':X''\rightarrow X'$
satisfying  $\beta''=\beta\circ |F|=\beta'\circ |F'|$.
Then  the equivalences $F$ and $F'$ induce,   via  the map ${((F')^\ast)}^{-1}\circ F^\ast$,   isomorphisms $H^\ast_{dR}(X)\rightarrow H^\ast_{dR}(X')$ and $H^\ast_{dR}(X,\partial X)\rightarrow H^\ast_{dR}(X',\partial X')$ which do not depend on the choice of $X''$. This is  a consequence of Lemma \ref{lps} and the definition of equivalent polyfold structures on the polyfold $Z$.  For more details  see \cite{HWZ3}.
Hence we can associate to $Z$ and $(Z,\partial Z)$
the connected simple systems, that is,  a category with precisely one morphism between any two (objects)  cohomology rings associated to the local models giving us natural identifications (of course,  the  morphism between any two  of the cohomology rings is a natural  isomorphism).  We denote these connected simple systems by $H^\ast_{dR}(Z)$ and $H^\ast_{dR}(Z,\partial Z)$ and call them de Rham cohomology of a polyfold $Z$ and the  pair $(Z, \partial Z)$, respectively.

\begin{thm}[{\bf Invariants in case of boundary}] \label{poil}
Let  $f$ be  a proper and  oriented Fredholm section of the strong
polyfold bundle $p:W\rightarrow Z$. Then there exists a well-defined
map
$$\Psi_f:H^\ast_{dR}(Z,\partial Z)\rightarrow {\mathbb R}$$
having the following properties.
Let $N$ be a fixed auxiliary norm on  the strong
polyfold bundle $p:W\rightarrow Z$. Assume that $U$ is an open neighborhood of
 the solution set  $S(f)=\{z\in Z\vert \, f(z)=0\}$  such that the pair  $(N,U)$ controls compactness.  Then for every  $\ssc^+$-multisection $\lambda$ with   its support in $U$ and such that  $N(\lambda )<1$  and $(f,\lambda)$  in good position, the solution set $S(f, \lambda )=\{z\in Z\vert \, \lambda (f(z))>0\}$ is a compact branched suborbifold with boundary with corners. Moreover,
$$
\Psi_f([\omega,\tau]) := \int_{(S(f,\lambda),\lambda_f)}\omega-
\int_{(\partial S(f,\lambda),\lambda_f)}\tau
$$
for every pair $[\omega , \tau ]$ representing a class in $H^n_{dR}(Z,\partial Z)$.
In addition, if $t\mapsto f_t$ is an sc-smooth, oriented, proper homotopy of sc-Fredholm sections, then
$$
\Psi_{f_0}=\Psi_{f_1}.
$$
\end{thm}
In the case that $\partial Z=\emptyset$ this simplifies to the following result.
\begin{thm}[{\bf Invariants in the case $\partial Z=\emptyset$}] \label{poil1}
Let $p:W\rightarrow Z$ be a strong polyfold bundle with $\partial
Z=\emptyset$ and let $f_0,f_1$ be two  proper and oriented Fredholm sections which
are properly and oriented homotopic. Then the associated maps
$$
\Phi_{f_0},\, \Phi_{f_1}:H^\ast_{dR}(Z)\rightarrow {\mathbb R}
$$
are the same.
\end{thm}
The results in \cite{HWZ3} show that there are enough (controlled)
$\ssc^+$-multi-sections to perturb a proper Fredholm problem
into good position.
\section{Local Branching Structures}
In this section  we shall study the local branching structures of a branched ep-subgroupoid $\Theta:X\to \Q^+$ of an ep-groupoid $X$.  As before we abbreviate the support of $\Theta$  by $S_{\Theta}=\supp \Theta=\{x\in X\vert \, \Theta (x)>0\}.$

\subsection{Good and Bad Points}
Our first goal is to show that
the tangent set $TS_{\Theta}$ of the support   $S_{\Theta}$ defined  in the introduction is independent of the branching structure used for its definition. To begin, let  $x\in S_{\Theta}$ and let $U=U(x)$ be an open neighborhood  of $x$ on which the
isotropy group $G_x$ acts by its  natural representation and  the restriction $\Theta\vert U$  is represented by the local branching structure $(M_i)_{i\in I}$ with the associated weights
$(\sigma_i)_{i\in I}$.  The local branches  $M_i$ are finite dimensional
 submanifolds of $X$ all of the same dimension, and
$$\supp \Theta \cap U=\bigcup_{i\in I}M_i.$$
We abbreviate $M_U:=\bigcup_{i\in I}M_i$. From the
 functoriality of $\Theta$ one concludes that
$$\varphi_g (M_U)=M_U\quad \text{for all $g\in G_x$}.$$
If we have another local branching structure $(N_j)_{j\in J}$
with weights $(\tau_j)_{j\in J}$ on $U$, then
$$
\supp \Theta \cap U=\bigcup_{j\in J}N_j=:N_{U}
$$
so that the  sets $M_U$ and $N_U$  are equal. However, looking at
the individual branches one cannot find, in general, for a given
branch $M_i$ a branch $N_j$ so that $M_i=N_j$. If $y\in M_U$, we
introduce the subset $I_y$ of the index set  $I$ as follows,
$$
I_y=\{i\in I\vert \, y\in M_i\}.
$$

\begin{defn}
A  point $x\in M_U$ is called a {\bf good point}
(with respect to the
local branching structure $(M_i)_{i\in I}$ ) provided  it possesses an open neighborhood $V=V(x)\subset X$ so that
$$V\cap M_i=V\cap M_j\quad \text{for all $i,j\in I_y$.}$$
\end{defn}
Next we shall prove  that being a good point is a property of the branched ep-subgroupoid $\Theta$ itself and not of the local branching structure used for its definition.

\begin{lem}\label{lemgoodind}
If $x\in M_U=N_U$ is a good point for the local branches $(M_i)_{i\in I}$,
then  it is also  a good point for
the local branches $(N_j)_{j\in J}$.
\end{lem}
\begin {proof}
Assume that $x\in M_U$ is a good point with respect
 to $(M_i)_{i\in I}$. Then $I_y=I_x$ for all points $y\in M_U$
 near $x$ and hence
$$\Theta (y)=\sum_{i\in I_y}\sigma_i=\sum_{i\in I_x}\sigma_i=\Theta (x).$$
Now take the second branching structure $(N_j)_{j\in J}$. Then for the given point $x\in M_U=N_U$, there exists a neighborhood $O(x)$ such that $J_y\subset J_x$ for all $y\in O(x) $ and hence
$$\Theta (y)=\sum_{j\in J_y}\tau_j\leq \sum_{j\in J_x}\tau_j=\Theta (x).$$
Because $J_y\subset J_x$ and $\tau_j>0$,  we conclude from the two statements that necessarily $J_y=J_x$ for all $y$ close to $x$ which implies the  assertion.
\end{proof}

Denote by  $\sg$ the  collection of all good points in the support
$S_{\Theta}=\supp \Theta$ of the branched ep-subgroupoid $\Theta:X\to \Q^+$. The  complement $\sba:=\st\setminus \sg$ consists of  {\bf bad points}.

\begin{lem}\label{lembadgood}
The set $\sg$ of good points is open and dense  in $S_{\Theta}$ and the set of bad points  $
\sba=S_{\Theta}\setminus \sg$ is closed and nowhere dense in $S_{\Theta}$.
\end{lem}
\begin{proof}
The set $\sg$ is open by  definition. Hence its complement $\sba$ is closed and we
show that it is nowhere dense.
Pick an arbitrary point $x\in \sba$  and a local branching structure $(M_j)_{j\in I}$ around $x$.  Not all manifolds $M_i$, $i\in I$,  near $x$  need to coincide. Define the
equivalence relation $i\sim_x j$ provided $M_i=M_j$ near $x$. Since
$x$ is a bad point,  we must have at least two equivalence classes.
Define $y_1=x$ and let $\alpha_1\subset I_{y_1}$ be an equivalence
class. Denote by $M_{\alpha_1}$ the local manifold representing
$\alpha_1$. The topology of $M_U$ is metrizable as a subset of the
metrizable $X$. Denote by $d$ a metric on $M_U$ and by $\abs{I}$ the cardinality of the index set $I$. Take $\varepsilon >0$. Then we find a point $y_2\not\in M_{\alpha_1}$  and $\varepsilon_1$  satisfying $d(y_2,y_1)<\varepsilon_1<\frac{\varepsilon}{\abs{I}}$.   There is   $\varepsilon_2<\frac{\varepsilon}{\abs{I}}$ so that every $z\in M_U$
with $d(z,y_2)<\varepsilon_2$ does not belong to $M_{\alpha_1}$. If
$y_2$ is a good point, we are done. Otherwise $y_2$ is a  bad  point and we define
again equivalence relation as before  where  $y_1=x$  is replaced by $y_2$. Since $y_2$ is a bad point, there are at least two equivalence classes.  Let $\alpha_2\subset I_{y_2}$ be one of them.  Then $y_2\in M_{\alpha_2}$. Note that the elements in all
the new equivalence classes do not contain any element in
$\alpha_1$. Now  we construct $y_3$ such that
$d(y_2,y_3)<\varepsilon_3<\frac{\varepsilon}{\abs{I}}$ which does not
belong to  $M_{\alpha_1}\cup M_{\alpha_2}$. After a finite number of
steps not exceeding $\abs{I}$ we find a point $y_k$ satisfying
$d(y_k,y_1)<\varepsilon$ which is a good point. This proves the desired
result.
\end{proof}

The next lemma shows  that the notion of being a good point is
compatible with morphisms.
\begin{lem}
Assume that $\varphi:x\to y$ is a morphism in ${\bf X}$ between two points $x, y$ in $ S_{\Theta}$. If $x$ is a good point, then  $y$ is also a good point.
\end{lem}
\begin{proof}
The  morphism $\varphi:x\to y$ extends to an sc-diffeomorphism $h=t\circ s^{-1}:V(x)\to V(y)$. This sc-diffeomorphism maps local branches  near $x$ onto local branches  around $y$.  This implies the result.
\end{proof}

It follows from the previous lemmata that being a good point is already defined on the orbit space $\abs{S_{\Theta}}\subset \abs{X}$ of the support of $\Theta$.

If $x\in \sg$ is a good point, then there exists a well defined tangent space $T_xS_{\Theta}$ because all the branches $M_j$ containing $x$ coincide near $x$. Hence, in view of Lemma \ref{lembadgood}, the support $S_{\Theta}$ has a tangent space at points in the complement of a closed nowhere dense set.  Moreover, if $\varphi:x\to y$ is a morphism between two good points, then it has an extension to a local sc-diffeomorphism $h:V(x)\to V(y)$.  The  tangent map $Th(x)$ at the point $x$, which we denote by
$T\varphi$,  maps  the tangent space $T_xS_\Theta$ bijectively onto the tangent space
$T_yS_\Theta$.

\subsection{The Tangent of the Support}
Having proved results about good and bad points, we are in a position  to prove that
the tangent set $TS_{\Theta}$ of the support   $S_{\Theta}$  is independent of the branching structure used for its definition.  We have just seen that at a good point $x\in S_{\Theta}$ the tangent space $T_xS_{\Theta}$ is a linear subspace of $T_xX$ which is independent of the local branching structure $(M_i)_{i\in I}$. If
$x\in S_{\Theta}$ is  a bad point, then $T_xS_\Theta$ is defined as the union of the tangent spaces $T_x M_i$ over $i\in I_x$. We fix a tangent space $T_xM_{i_0}$ and let $(N_j)_{j\in J}$ be a second local branching structure around the bad point $x$. By  Lemma \ref{lembadgood}, there exists a sequence $(x_k)$ of good points on the manifold $M_{i_0}$ converging to the bad point $x$. Due to Lemma \ref{lemgoodind},  the points $x_k$ are also good points of $N_{j_k}$ for some $j_k\in J.$ After taking a subsequence, we may assume that $j_k=j_0$ is constant in $k$ and hence $T_{x_k}N_{j_0}=T_{x_k}M_{i_0}$. As $x_k\to x$,  we  conclude  $T_xN_{j_0}=T_xM_{i_0}$. This shows that the definition of $TS_{\Theta}$ is indeed independent of the choice of a local branching structure.

As a side remark we observe that there is another way of looking at the construction of the tangent $TS_{\Theta}$. Let  ${(M_i)}_{i\in I}$  with the associated weights  $(\sigma_i)_{i\in I}$  be  a local branching structure  for $\Theta$ covering $\abs{M_U}$ and let $\tau:TX\rightarrow X^1$ be  the tangent projection.
Define $T\Theta:TX\rightarrow {\mathbb Q}^+$ by $T\Theta(h)=\Theta(x)$ where
$\tau(h)=x$ and  $h\in T_xM_i$ for some $i\in I$. Otherwise,  define $T\Theta(h)=0$. This definition does not depend on the choices involved. Then $T\Theta$ is a branched ep-subgroupoid of $TX$ and
$$
S_{T\Theta}=TS_\Theta
$$
as previously defined. The set of linear local branches $(TM_i)_{i\in I}$ with the same weights $(\sigma_i)_{i\in I}$ defines  a local branching structure for $TS_\Theta\vert \abs{M_U}$.

\subsection{Essential Points  of Local Branches}\label{essentialp}
Next we shall study  the relationship between two local
branching structures on the open neighborhood $U(z_0)\subset X$ of the point $z_0\in \supp \Theta$. We assume that on $U=U(z_0)$ we have the natural representation $g\mapsto \varphi_g \in \text{Diff}_{\ssc}(U)$ of the isotropy group $G_{z_0}$ and a local  branching structure $(M_i)_{i\in I}$ with the weights $(\sigma_i)_{i\in I}$. As usual   $M_U=\bigcup_{i\in I}M_i=\supp \Theta \cap U$.

We denote  by $TM_U$ the restriction of the tangent set $T\st$ to the points in $M_U$. We know that $\varphi_g (M_U)=M_U$ for all $g\in G_{z_0}$ and we have  an induced map $T\varphi_g :TM_U\to TM_U$. Any morphism $\varphi:x\to y$ in ${\bf X}$ between two points $x, y$ in $U$ is of the form $\varphi=\Gamma (g, x)$ for a unique $g\in G_{z_0}$ and $\varphi_g(x)=y$.

\begin{defn} Let $M$ be a smooth $n$-dimensional manifold with boundary with corners and let $K$ be a subset of  $M$. Then the tangent space $T_zM$ at a point $z\in K$ has a {\bf distinguished linear subspace}, denoted by $T_z^KM\subset T_zM$, and characterized by the following property. Given any chart $\varphi:U(z)\subset M\to O(0)\subset [0,\infty )^d\times \R^{n-d}$ satisfying $\varphi (z)=0$, then $T_z^KM$ is the subspace of $T_zM$ such that $T\varphi  (z)(T^K_zM)\subset \R^n$ is the linear hull of all unit vectors $e\in \R^n$ for which there exists a sequence $z_k\in K\setminus \{z\}$ satisfying
\begin{itemize}
\item[$\bullet$] $\lim z_k=z$
 \item[$\bullet$] $\lim \dfrac{\varphi (z_k)}{\abs{\varphi  (z_k)}}=e$.
 \end{itemize}
If there is no such sequence,  then we put $T_z^KM=\{0\}$.
\end{defn}

The definition of the subspace $T^{K}_zM$ is independent of the choice of  the chart $\varphi$. If $(M_j)_{j\in I}$ is the local branching structure of $U$,
we associate with a triple $(i, g, j)$,
in which $i, j\in I$ and $g\in G_{z_0}$,
the closed subset $K (i, g, j)$ of $M_i$ defined by
$$
K{(i,g,j)}=\{x\in M_i\vert \, \varphi_g(x)\in M_j\}.
$$

\begin{defn}
Let $(M_i)_{i\in I}$ with the weights $(\sigma)_{i\in I}$  be a local  structure for $\Theta$ in
the open set $U$.
\begin{itemize}
\item[(1)]
A  point $x\in M_i$ is {\bf strongly  $(i,g,j)$-essential}, if there exists a sequence
$(x_k)\subset K{(i,g,j)}\subset M_i$ satisfying
\begin{itemize}
\item[$\bullet$] $\lim x_k=x$.
\item[$\bullet$] $T^{K{(i,g,j)}}_{x_k}M_i=T_{x_k}M_i$ for all $k$.
\item[$\bullet$] $T^{K{(j,g^{-1},i)}}_{\varphi_g(x_k)}M_j=T_{\varphi_g(x_k)}M_j$ for all $k$.
\end{itemize}
\item[(2)]  A  point $x\in M_U$ is called {\bf $(i,g,j)$-essential}, if there exist finite  sequences   $i_0=i,i_1,\ldots ,i_k=j$ and $g_0,\ldots , g_{k-1}\in G_{z_0}$  such that
$g=g_{k-1}\circ \cdots \circ g_0$ and if $x_0=x$,  $x_l=\varphi_{g_{l-1}}(x_{l-1})$  for $1\leq l\leq k$, then the points $x_l$  are strongly  $(i_l,g_l,i_{l+1})$-essential.
\end{itemize}
\end{defn}

We point out that if $x\in M_i$ is $(i, g, j)$-essential,
then $\varphi_g (x)\in M_j$ is $(j, g^{-1}, i)$-essential.
Moreover, if $x\in M_i$ is $(i, g, j)$-essential and
 if $\varphi_g (x)$ is $(j, h, l)$-essential, then $x$ is
 $(i,h\circ g,l)$-essential.

If $x\in  K(i, g, j)$, then $\varphi_g (x)\in M_j$ and the tangent map $T_{x}\varphi_g :T_xX\to T_{\varphi_g (x)}X$ is an isomorphism. However,  this does not imply that
$T_{x}\varphi_g  (T_xM_i)\subset T_{\varphi_g (x)}M_j$.
As the next lemma shows,  this holds true for $(i, g, j)$-essential points.

\begin{lem}\label{newlem}
If $x\in M_i$ is  $(i, g, j)$-essential, then
$T_{x}\varphi_g  (T_xM_i)\subset T_{\varphi_g (x)}M_j$ and,  in particular for dimension reasons,
the restriction
$$T\varphi _g (x)\vert T_xM_i:T_xM_i\to T_{\varphi_g (x)}M_j$$
is a linear isomorphism.
\end{lem}
\begin{proof}
First, we  prove  the conclusion of the lemma under the additional  assumption that
$T_x^{K(i, g, j)}M_i=T_xM_i$. The set $K(i,g,j)$ is bijectively mapped onto the set  $K(j,g^{-1},i)$  by the map $\varphi_g$. Since $M_j$ is a submanifold of $X$,  there exists a smooth chart
$\psi:O(M_i,x)\rightarrow O([0,\infty)^k\times {\mathbb R}^{n-k},0)$ so that
$\psi^{-1}:O([0,\infty)^k\times {\mathbb R}^{n-k},0)\rightarrow X$ is sc-smooth. Then the composition $\varphi_g\circ \psi^{-1}:O([0,\infty)^k\times {\mathbb R}^{n-k},0)\rightarrow X$ is sc-smooth. By assumption,  $T_x^{K(i, g, j)}M_i=T_xM_i$. Hence there exist at least $n$ linearly independent unit vectors $a_1,\ldots ,a_n$ in ${\mathbb R}^n$
so that for any of them, say $a=a_l$,  there exists a sequence $z_k$ in
$\psi(K(i,g,j)\cap O(M_i,x))$ so that $z_k\rightarrow 0$ and $\frac{z_k}{|z_k|}\rightarrow a$.
Then  $\varphi_g\circ\psi^{-1}(z_k)$ is a sequence in $M_j$ converging to $y=\varphi_g(x)$. By construction,
$$
T\varphi_g(x)\circ T\psi^{-1}(0)a\in T_yM_j.
$$
Since $T\psi^{-1}(0)a$ belongs to $T_x^{K(i,g,j)}M_i$,  we conclude that
$T_xM_i$ has a basis mapped by $T\varphi_g(x)$ into $T_{\varphi_g(x)}M_j$.
The  map $T\varphi_g(x)$ is an isomorphism and the submanifolds $M_i$ and $M_j$ are assumed to have  the same dimension $n$. Consequently,
\begin{itemize}
\item if $T_x^{K(i, g, j)}M_i=T_xM_i$,  then $T\varphi_g(x):T_xM_i\rightarrow T_{\varphi_g(x)}M_j$ is an isomorphism.
    \end{itemize}
Next, we assume that $x\in K(i,g,j)$  is strongly $(i,g,j)$-essential. Hence $x$ is the limit of a sequence $x_k\in K(i,g,j)$ so that for every $x_k$ we have $T_{x_k}^{K(i, g, j)}M_i=T_{x_k}M_i$. It follows from the sc-smoothness of $\varphi_g$,  after  passing to the limit,
that $T\varphi_g(x)$ maps $T_xM_j$ into $T_{\varphi_g(x)}M_j$ and we draw the same conclusion as before. Hence we have proved
\begin{itemize}
\item if $x$ is strongly $(i,g,j)$-essential then $T\varphi_g(x):T_xM_i\rightarrow T_{\varphi_g(x)}M_j$ is an isomorphism.
    \end{itemize}
Finally, assume that $x$ is only $(i,g,j)$-essential. By the definition of an $(i,g,j)$-essential point, there are finite sequences  $i_0=i,i_1,\ldots ,i_k=j$ and $g_0,\ldots, g_{k-1}\in G_x$ such that $g=g_{k-1}\circ\cdots \circ g_0$ and if $x_0=x$ and $x_l=\varphi_{g_{l-1}}(x_{l-1})$ for $1\leq l \leq k$, then  the points $x_l$ are strongly $(i_l,g_l,i_{l+1})$-essential.
By the previous discussion
$$
T\varphi_{g_{l-1}}(x_{l-1}):T_{x_{l-1}}M_{i_{l-1}}\to
T_{x_{l}}M_{i_{l}}
$$
is an isomorphism.  Therefore,  the linearization of the map
$\varphi_g=\varphi_{g_{k-1}}\circ\cdots \circ \varphi_{g_0}$ at  $x$  is a composition of isomorphisms and the result follows.
\end{proof}

We  denote by  $\ke(i, g, j)$  the collection of all points $x\in K(i, g, j)$ which are $(i,g, j)$-essential. The complement
$$\kin (i, g, j):=K(i, g, j)\setminus \ke (i, g, j).$$
consists of {\bf inessential points}.

Now assume  that the branched ep-subgroupoid $\Theta:X\to \Q^+$ is $n$-dimensional. Then a   subset $K\subset M_j$ is said to be of  {\bf measure zero}, if for every smooth chart $\varphi :U\subset M_j\to \R^n$, the set $\varphi (U\cap K)\subset \R^n$ is of $n$-dimensional  Lebesgue measure zero.

\begin{lem}\label{lem0}
For every triple $(i,g,j)$,  the set $\kin (i,g,j)\subset M_i$ of inessential points is  open in
$K(i,g,j)$ and of  measure zero.
\end{lem}
\begin{proof}
By construction, the set $\ke (i, g, j)$ of essential points is closed in $K (i, g, j)$.
Hence  its complement  is open in $K(i,g,j)$. By the definition of being
$(i,g,j)$-essential,  it suffices to prove  that the set of points in
$M_i$ which are  not strongly $(i,g,j)$-essential is  of measure zero.
To do  so we
pick a point $x\in \kin (i,g,j)$ which is not strongly
$(i,g,j)$-essential and choose  a chart $\varphi:V(x)\subset
M_i\rightarrow O(0)\subset [0,\infty)^d\times {\mathbb R}^{n-d}$ satisfying $\varphi (x)=0$. We abbreviate the image set in $\R^n$ by $K=\varphi [\kin (i, g, j)\cap V(x)]$. It suffices to show that $0$ is a point of Lebesgue density $0$, i.e.,
$$
\lim_{\varepsilon\rightarrow 0}
\frac{1}{\varepsilon^n}\mu(B_\varepsilon\cap K)=0,
$$
where $\mu$ stands for the $n$-dimensional   Lebesgue measure  and  $B_{\varepsilon}$ is the  ball of radius $\varepsilon$ centered at the origin. Since $x$ is not strongly $(i, g, j)$-essential, we have $T_x^{K(i, g, j)}M_i\neq T_xM_i$. Hence by
composing the chart with a rotation in the image,  we may assume
that
$$\Sigma:=T\varphi(x)(T_x^{K(i,g,j)}M_i)\subset\R^{n-1}\times \{0\}.$$
Now  take a number $\delta>0$, define the set
$\Gamma_\delta=\{(a, b)\in \R^{n-1}\times \R\vert \, \abs{b}\leq \delta \abs{a}\}$ and consider the subset
$(B_{\varepsilon}\setminus \Gamma_{\delta})\cap K$ of  $\R^n$. If, for a given a sequence $\varepsilon_k\to 0$,  there exists a sequence of points $x_k$ in
$(B_{\varepsilon_k}\setminus \Gamma_{\delta})\cap K$, then  we arrive at a contradiction to the definition of $\Sigma$. Consequently,   the set
$B_{\varepsilon}\cap K$  is contained in $\Gamma_{\delta}$ if $\varepsilon$ is sufficiently small. Hence
\begin{equation*}
\begin{split}
\limsup_{\varepsilon\to  0}\frac{1}{\varepsilon^n}\mu(B_\varepsilon\cap K)\leq
\limsup_{\varepsilon\rightarrow
0}\frac{1}{\varepsilon^n}\mu(B_\varepsilon\cap \Gamma_\delta)\leq
C(\delta)
\end{split}
\end{equation*}
for a constant $C(\delta)$ satisfying $C(\delta)\rightarrow 0$ as $\delta\rightarrow 0$. This shows that
$$
\lim_{\varepsilon\rightarrow 0}\frac{1}{\varepsilon^n}\mu(B_\varepsilon\cap K) =0.
$$
We have proved that the Lebesgue density at every point $x$ in the Borel set $\kin (i, g, j)\subset M_i$ vanishes. This implies that this  set is of measure zero.
\end{proof}
\begin{prop}\label{prop2.9}
If $x\in M_i$ is $(i, g, j)$-essential, then there exists a smooth map $h:O(x)\subset M_i\to M_j$ satisfying
\begin{itemize}
\item[$\bullet$]  $h(y)=\varphi_g(y)$ if  $y\in O(x)\cap K(i,g,j)$.
\item[$\bullet$] $Th (x)=T\varphi_g(x)\vert T_xM_i$.
\end{itemize}
\end{prop}
\noindent{{\it Proof.}}\
Assume $x\in M_i$ is  $(i,g,j)$-essential. Then we find finite sequences $i_0=i,i_2, \ldots ,i_k=j$ and $g_0, \ldots ,g_{k-1}\in G_{z_0}$ such that $g=g_{k-1}\circ \cdots \circ g_0$ and  $x_0=x$  is strongly $(i, g_0, i_1)$-essential and  the point $x_l=\varphi_{g_{l-1}}(x_{l-1})$ is  strongly $(i_{l},g_{l},i_{l+1})$-essential for $1\leq l\leq k-1$.
By Lemma \ref{newlem} the tangent map $T\varphi_g(x)$ of the sc-diffeomorphism $\varphi_g:U\to U$  induces a linear isomorphism
$$
T\varphi_g(x):T_xM_i\rightarrow T_{\varphi_g(x)}M_j.
$$
Hence the submanifolds $\varphi_g (M_i)$ and  $M_j$ are tangent at the point
$\varphi_g(x)\in M_j$. Using   compositions of mappings,  it suffices to prove the proposition at a point $x$ which is strongly $(i,g,j)$-essential for some $x\in M_i$ and some $g\in G_{z_0}$. Since  $T\varphi_g (x):T_xM_i\to T_{\varphi_g (x)}M_j$ is  an isomorphism,  we may replace  $M_j$ by $\varphi_g^{-1}(M_j)=M'_j$ and  arrive at the situation that $T_xM_i=T_xM'_j$. If  we can construct a smooth extension $O(x)\subset M_i\to O'(x)\subset  M'_j$ of the identity,  then the composition with $\varphi_g$ has the desired properties. Therefore, we may assume without loss of generality that $T_xM_i=T_xM_j$.  Now  the  statement  of the proposition is a consequence of the following lemma  where we use the notation and concepts introduced in \cite{HWZ1} and \cite{HWZ2}.
\begin{lem}
 Let  ${\mathcal S}=(\pi,E,V)$ be a splicing and $O$ be an open
neighborhood of $0$ in the associated splicing core $K^{\mathcal S}$. Assume that $M$
and $M'$ are  finite dimensional submanifolds   of $O$  containing $0$ and parameterized
by graphs over the tangent spaces $T_0M$ and $T_0M'$ which are in good position to the partial quadrant $C\oplus E$. Finally, assume that  $T_0M=T_0M'$. Then there exist an
open neighborhood of  $O(0)$ of $0\in M$ and a smooth map  $h:O(0)\subset M\to M'$ satisfying  $h(0)=0\in M\cap M'$ and $Th(0)=\text{id\ }\vert T_0M$.
\end{lem}
\begin{proof}[Proof of Lemma]
The tangent spaces $T_0M=T_0M'$  are equal  to the finite dimensional subspace $N$ of an sc-smooth Banach space $W\oplus E$  which is  in a good position to the partial quadrant $C\oplus E$ in $W\oplus E$.  Moreover, there is an sc-smooth complement $N^{\perp}$ of $N$ so that $W\oplus E=N\oplus N^{\perp}$. The manifolds  $M$  and $M'$ are  represented locally as  the graphs  of the sc-smooth maps
$\Gamma$ and $\Gamma'$ defined on an open neighborhood $Q$ of $0$ in $C\cap N$ into $N^{\perp}$. They are  given by $\Gamma (p)=p+A(p)$ and $\Gamma' (q)=q+B(q)$ with the maps $A, B:Q\to N^{\perp}$ satisfying  $A(0)=B(0)=0$ and $DA(0)=DB(0)=0$.   Now define  the smooth map $h:M\to N$ locally near $0$ by
$$h(q+A(q))=q+B(q).$$
This is an extension of the identity map $M\cap M'\subset M\to M'$ which satisfies $h(0)=0$ and $Dh(0)=\text{id}$ on $T_0M=T_0N$. The  proof of the lemma and hence  of the proposition is complete.
\end{proof}

Let $(M_i)_{i\in I}$ with the associated weights $(\sigma_i)_{i\in I}$ be   the local branching structure on the open set $U=U(z_0)\subset X$, and let $M_U:=\bigcup_{i\in I}M_i \subset U$. For  every point $x\in M_U$, define the subset $I_x$ of the index set $I$ by $I_x=\{i\in I\vert \, x\in M_i\}$. We observe that the point $x\in M_U$ has an open neighborhood $O(x)$ in $M_U$ such that $I_y\subset I_x$ for  all $y\in O(x)$. If $x\in M_U$,  we define a partion $P_x$ of $I_x$ into equivalence classes as follows. If $i$ and $i'$ belong to the index set $I_x$, then
$$\text{$i\sim i'$ if and only if $x$ is $(i, \text{id}, i')$-essential}.$$
 Here $\text{id}$ stands for the identity element of the isotropy group $G_{z_0}$ and hence $\varphi_{\text{id}}$ is the identity map on $U$. In particular,  by Lemma \ref{newlem},
$$T_xM_i=T_xM_{i'}$$
if $x$ is $(i, \text{id}, i')$-essential. We  denote by ${\mathcal P}$  the set of all index pairs $(J, P)$ where $J$  is a subset of the index set
$I$ and where $P$ is any partition of $J$. Then we denote by $M^{J, P}_U$  the subset of $M_{U}$ defined by
$$
M_U^{J,P}=\{x\in M_U\vert \, \text{$I_x=J$ and $P_x=P$}\}.
$$
The collection of all the subsets $M^{J, P}_U$ where $(J, P)\in {\mathcal P}$ constitutes a partition of $M_{U}$ into finitely many Borel subsets.

\begin{lem}\label{lemA}
If the partition $P$ of $J$ contains more than one element, then
the set $M_U^{J, P}$ is of measure zero.
 \end{lem}
\begin{proof}
The set $M_U^{J,P}$ consists of all points $x\in M_U$ such that  $I_x=J$ and, if  $\gamma\in P$ and $i, j\in \gamma$, then $x$   is $(i,\text{id},j)$-essential. Assume that $P$ contains at least two elements $\gamma_1$ and $\gamma_2\in P$.  Then we conclude for $i\in \gamma_1$ and $j\in\gamma_2$ that
$x$ is not $(i,\text{id},j)$-essential.  Consequently, in view of Lemma \ref{lem0}, the set $M_U^{J, P}$ is a subset of a set of measure zero and hence a set of measure zero. \end{proof}

Define the subset  $M_U^\ast$ of $M_U$  by
$$M_U^\ast=\{x\in M_U\vert \,  \text{$x\in M_U^{J, P}$ and $P=\{J\}$}\}.$$
If $P=\{J\}$, the partition  of the index set $J$ consists of the single element $J$. By  Lemma \ref{lemA}, the set $M_U\setminus M_U^*$ is of measure zero.

Next consider a second local  branching structure  of $\Theta$ on $U$ given by the branches $(N_{i'})_{i'\in I'}$ and  the weights $(\sigma_{i'})_{i'\in I'}$   and denote by $x\mapsto I_x'$ and $x\mapsto P_x'$ the associated index maps. Then $M_U=N_U=\supp \Theta \cap U$ and we can combine both local branching structures by
 halving their weights to get a third local branching structure $(M''_j)_{j\in I''}$ and $(\delta_j)_{j\in I''}$ for $\Theta$ on $U$ as follows. The index set $I''$ is the disjoint union  $I''=I\cup I'$. The branches are defined as $M''_i=M_i$  if $i\in I$ and $M''_{i'}=N_{i'}$ if $i'\in I'$. The associated weights $(\delta_j)_{j\in I''}$ are defined as
 $$\delta_{i}=\frac{1}{2}\sigma_i\quad \text{ and}\quad
\delta_{i'}=\frac{1}{2}\tau_{i'}$$ for $i\in I$ and $i'\in I'$. For
this new branching structure of $\Theta$ on $U$,  we have the index
maps $x\rightarrow I_x''=I_x\cup I_x'$ and $x\rightarrow P_x''$
defined on $M_U=N_U=M_{U}''$. If $J''$  is a subset of $I''$ and
$P''$ is a partition of $J''$,  we set
$$M_U^{J'', P''}=\{x\in M_U\vert \, \text{$I_x''=J''$ and $P_x''=P''$}\}.$$
Since $I_x''=I_x\cup I_x'$, we conclude from $I_x''=J''$  and $J''=J\cup J'$ for the index sets $J\subset I$ and $J'\subset I'$, that $I_x=J$ and $I'_x=J'$ for all $x\in M_U^{J'', P''}$.
Moreover, there are two partitions $P$ of $J$ and $P'$ of $J'$ such that $P_x=P$ and $P_x'=P'$ for all $x\in M_U^{J'', P''}$.  This observation will be very useful in
Section \ref{subsect3.2} and we formulate it as a theorem.

\begin{thm}[{\bf Index maps}]\label{thmindex}
 For every index pair $(I'',P'')$,
the index maps $x\mapsto I_x$ and $x\mapsto  I_x'$, $x\mapsto P_x$ and $x\mapsto P_x'$ are constant as $x$ varies in the sets $M_U^{I'', P''}$.
\end{thm}

\section{Branched Integration}

In this section we shall define the canonical $\sigma$-algebra ${\mathcal L}(S)$ and define the measure on ${\mathcal L}(S)$. Throughout this section we assume that $\Theta:X\rightarrow
\Q^+$  is an $ n$-dimensional   branched ep-subgroupoid of the ep-groupoid $X$ and assume that the orbit space  $S=\abs{\supp(\Theta)}$ of the support of $\Theta$ is compact.

\subsection{The Canonical $\sigma$-Algebras and Theorem \ref{th0}}\label{subsec3.1}

We begin by  defining  the $\sigma$-algebra ${\mathcal L}(S)$ of measurable subsets of  $S$,  called  the  canonical $\sigma$-algebra.

Every point $x\in \supp \Theta$ has an open neighborhood $U(x)$ on which the isotropy group $G_x$ acts by the natural $G_x$-action, and which contains the local branches $(M_i)_{i\in I}$. We  choose open  subsets $V(x)$  invariant under the natural $G_x$-action and satisfying $\ov{V(x)}\subset U(x)$. Projecting all the open sets $V(x)$ down to the orbit space $S$,  we obtain the open cover $\{\abs{V(x)}\}$ of $S$. By assumption  the orbit  space $S$ is compact and  we find finitely many points $x_{\alpha}$, $\alpha=1,\ldots, ,k$,  such that the sets $\abs{V(x_{\alpha})}$  cover $S$.
Abbreviate  $V_{\alpha}=V(x_{\alpha})$ and $U_{\alpha}:=U(x_{\alpha})$. Every set $U_{\alpha}$  contains  the branches $(M^{\alpha}_i)_{i\in I^{\alpha}}$  which are properly embedded finite dimensional submanifolds of $X$ of dimension $n$. Therefore, there are finitely many mappings
$$M_i^{\alpha}\to S, \quad x\mapsto |x|$$
into the orbit space $S$ for $\alpha=1,\ldots ,k$ and $i\in I^{\alpha}$.  If $A$ is a subset of $S$,  we introduce the sets
$$A_{\alpha}=A\cap \abs{V(x_{\alpha})}\subset S$$
and denote by $A^{\alpha}_i\subset M_i^{\alpha}$ the preimage of $A_{\alpha}$ in  the manifold $M_i^{\alpha}$ under the projection map $M_i^{\alpha}\to S$ into the orbit space given by $x\mapsto |x|$.

\begin{defn}
The set $A\subset S$ is of {\bf measure zero} if all the sets $A_{i}^{\alpha}\subset M_i^{\alpha}$ are of (Lebesgue) measure zero  as subsets of the  $n$-dimensional smooth manifolds $M_i^{\alpha}$.
\end{defn}

\begin{lem}
The definition of a set of measure zero does not depend on the choice of
the local branching structures.
\end{lem}
\begin{proof} It suffices to consider a set $A\subset S$ which is contained in $\abs{V}$ where
$V$ is an open set in $X$ such that $V\subset \ov{V}\subset U$ and $U$ is invariant with respect to  its natural action and contains  two different local branching structures,
say $(M_i)_{i\in I}$ and $(N_j)_{j\in J}$ having the associated weights $(\sigma_i)_{i\in I}$
and $(\tau_j)_{j\in J}$.  Then
$$M_U=\bigcup_{i\in I}M_i=N_U=\bigcup_{j\in J}N_j.$$
Now assume  that $A\subset S$ is  of measure  zero with respect to the branching structure $(M_i)_{i\in I}$ and consider the preimages $A_i\subset M_i$ of $A\cap \abs{V}$ under the map $M_i\to \abs{M_U}\subset S$.  Then, if $\wt{A}\subset M_U$ is the preimage of $A\cap \abs{V}$ under the map $M_U\to \abs{M_U}$, we have
$$A_i=\wt{A}\cap M_i=\bigcup_{j\in J}(\wt{A}\cap M_i\cap N_j).$$
If  $A_i\subset M_i$ has measure zero, then also $\wt{A}\cap M_i\cap N_j\subset M_i$ has measure zero. Since the identity map $M_i\cap N_j\subset M_i\to M_i\cap N_j\subset N_j$ has a smooth extension to open neighborhoods $O(M_i\cap N_j)\subset M_i\to O(M_i\cap N_j)\subset N_j$, we conclude that the set $\wt{A}\cap M_i\cap N_j$ has measure zero also as a subset of $N_j$. This holds for every $i\in I$. Hence $\wt{A}\cap N_j=\bigcup_{i\in I}(\wt{A}\cap M_i\cap N_j)$ has measure zero on the branch $N_j$. This holds true for every $j\in J$ and since $A_j=\wt{A}\cap N_j$ is the preimage of $A\cap \abs{V}$ under the map $N_j\to \abs{N_U}=\abs{M_U}$, we conclude that $A$ is of measure zero  with respect to the second branching structure $(N_j)_{j\in J}$. The argument shows that the definition of a set of  measure zero  is independent of the choice of the branching structures used.
\end{proof}

\begin{defn} [{\bf Canonical $\sigma$-algebra ${\mathcal L }(S)$}]
Let $X$ be an ep-groupoid and let
 $S=\abs{\supp \Theta}$ be the compact orbit space of the branched ep-subgroupoid $\Theta:X\to \Q^+$. We define the collection ${\mathcal L}(S)$ of measurable subsets of $S$  as  the smallest
$\sigma$-algebra  which contains the  Borel $\sigma$-algebra ${\mathcal B}(S)$  and all  subsets  of $S$ of measure zero  as just defined.
\end{defn}

In other words,  in view of the definition of a set of measure zero, the $\sigma$-algebra ${\mathcal L}(S)$ is obtained from the Borel $\sigma$-algebra ${\mathcal B}(S)$ by adding all the subsets of sets in ${\mathcal B}(S)$ of measure zero.

Recall that the  boundary $\partial S$  of the orbit space $S=\abs{\supp \Theta}$ is  defined as the subset $\partial S=\{|x|\in S\vert \, x\in \supp \Theta \cap \partial X\}$.  To  define the canonical $\sigma$-algebra ${\mathcal L}(\partial S)$,  we first make the following observation.  Assume that $M$ is an $n$-dimensional smooth manifold with boundary $\partial M$ with corners and consider a boundary point $x\in \partial M$ having degeneracy index $d(x)=d$.  In local coordinates we may view  $x$ as the point $0$ in $[0,\infty)^d\times
{\mathbb R}^{n-d}$. Near $0$, the boundary $\partial M$ of $M$ consists of $d$-many faces. If $d\geq 2$, then the intersection set of two faces is, as a subset of any of two faces, a set of $(n-1)$-dimensional  Lebesgue  measure zero. In view of this observation, the construction of the $\sigma$-algebra ${\mathcal L}(\partial S)$ proceeds analogously to the construction of ${\mathcal L}(S)$ above.

\subsection{Branched integration and Theorem \ref{th1}}\label{subsect3.2}
After defining the canonical $\sigma$-algebra ${\mathcal L}(S)$, we shall define a signed  measure on measurable subsets of $S$. Since we are going to use sc-smooth partitions of unity, we assume that
the  sc-structure on the ep-groupoid is based on separable sc-Hilbert
spaces. This has the advantage that $X$ admits sc-smooth partitions  of
unity.  The compact orbit space $S= \abs{\supp \Theta}$
is equipped with the weight function $\vartheta: S\to
\Q^+\cap (0, \infty )$ defined  by $\vartheta (|x|):=\Theta(x)$ for $x\in \supp \Theta$. The aim is to define,  for a given sc-differential $n$-form $\omega$ on $X$ ($X^i$ or $X_\infty$),  a signed measure
$\mu_{\omega}^{S, \vartheta}$ on the measure space $(S, {\mathcal
L}(S))$.  The strategy is to use the local branching structures and
to show that  the construction does not depend on the choices.

We start with the local branching structure $(M_i)_{i\in I}$ with the weights $(\sigma_i)_{i\in I}$ in the open neighborhood  $U=U(x)$ of the point $z_0\in \supp \Theta$ on which the isotropy group $G=G_{z_0}$ acts by the natural representation $\varphi:G\to \text{Diff}_{\ssc}(U)$. As before $M_U=\bigcup_{i\in I}M_i$ is  the union of all the branches in $U$. If $G_0$ is the ineffective part of $G$, we denote the order of the effective part $G_{\text{e}}:=G/G_0$ by $\sharp G_{\text{e}}$.

\begin{defn}
Let $K\in {\mathcal L}(S)$ be a measurable set on $S$  contained in a compact subset of $\abs{M_U}$. Then we define the real number $\mu_{\omega}^U(K)$ by
$$\mu_{\omega}^U(K):=\dfrac{1}{\sharp G_{\text{e}}}\sum_{i\in I}\sigma_i \int_{K_i}\omega\vert M_i.$$
\end{defn}
Here $K_i\subset M_i$ is the preimage of $K$ under the projection map $M_i\to \abs{M_U}$ defined by $x\mapsto |x|$. The integral $\int_{K_i}\omega\vert M_i$ stands for the signed measure of $K_i$ associated with a smooth $n$-form on an oriented $n$-manifold. It is defined as
$$\int_{K_i}\omega\vert M_i=\lim_k \int_{U_k}j^*\omega$$
where $j:M_i\to X$ is the sc-smooth inclusion mapping and $(U_k)$ is a decreasing sequence of open neighborhoods of the set $K_i\subset M_i$ satisfying $\bigcap_k U_k=K_i$.

Note that for every compact subset $C\subset \abs{M_U}$ we have  signed measure on ${\mathcal L}(S)$ defined by
$$K\mapsto \mu^{U}_{\omega}(C\cap K).$$

In order to demonstrate that $\mu_{\omega}^U(K)$ does not depend on the choice of the branching structure $(M_i)_{i\in I}$ and $(\sigma_i)_{i\in I}$, we take a second branching structure  in $U$ consisting of  the branches $(N_j)_{j\in J}$ and the associated weights
$(\tau_j)_{j\in J}$ and denote the union of the branches by $N_U=\bigcup_{j\in J}N_j$. Then $M_U=N_U$.
\begin{lem}[{\bf Independence}]\label{lemB}
$$
\sum_{i\in I}\sigma_i \int_{K_i}\omega|M_i=
\sum_{j\in J}\tau_j \int_{L_j}\omega|N_j,
$$
where $L_j$ is the preimage in $N_j$ of $K$ under the projection map $N_j\to \abs{N_U}=\abs{M_U}$.
\end{lem}
\begin{proof}
We first observe that taking a finite partition of $|M_U|=|N_U|$, say $\abs{M_U}=\bigcup
\Gamma_\alpha$,  it suffices to prove the lemma for   measurable subsets $K$  in  $\Gamma_\alpha$.  In order to obtain a convenient partition,  we take a third branching structure which is the union of the two branching structures with halved
weights as described at the end of Section \ref{essentialp}. The index set is the disjoint union $I''=I\cup J$. The set $M_U$ is partitioned into the sets
$$M^{A, P}_U=\{x\in M_U\vert \,  \text{$I_x''=A$ and $P_x''=P$}\}$$
where $P$ is a partition of the index set $A\subset I''$ and where $P_x''$ is the partition of $I_x''$ into its equivalence classes as described above. Then the sets
$\Gamma_{A, P}=\abs{M_U^{A,P}}$ define a partition of $\abs{M_U}$.
By   Lemma \ref{lemA}  the set  $\Gamma_{A, P}\subset S$ is  of measure zero if $P\neq \{A\}$. Therefore, we may assume that $P=\{A\}$. Setting
$\Gamma_A=\Gamma_{(A,\{A\})}$ we then have, up to a set of measure $0$ in $S$, the partition
$$
\abs{M_U}=\bigcup_{A\subset I''} \Gamma_A.
$$
As explained at the beginning of the proof we may now assume that $K\subset \Gamma_A$ for some index set $A\subset I''$. In view of the definition of $M^{A, \{A\}}_U$, we conclude from Theorem \ref{thmindex}, that  the sets $I_x$  and $J_x$ are constant for all
$x\in M^{A, \{A\}}_U$, say  $I_x=I_K$ and $J_x=J_K$. We shall denote by $K_i\subset M_i$ if $i\in I_K$ and by $K_j\subset M_j$ if $j\in J_K$ the preimages of $K$ under the projection maps $M_i\to \abs{M_U}$ and $N_j\to \abs{M_U}$.
By the definition of the partition of the index set $A=I_K\cup J_K$ into equivalence classes of essential points, we know that $i\sim i'$ for all $i, i'\in I_K$ and $j\sim j'$ for all $j, j'\in J_K$ and also $i\sim j$ for all $i\in I_K$ and $j\in J_K$. Therefore, the sets $K_i$ and $K_j$ are all the same as subsets of $M_U$, but may lie  in possibly different manifolds $M_i$ and $N_j$. By construction, the points $x\in K_i$ under consideration are all essential points and we conclude from Proposition \ref{prop2.9} (in  the special case  $g=\text{id}$) that the identity map  $K_i\subset M_i\to K_{i'}\subset M_{i'}$,  for $i$ and $i'\in I_K$, has a smooth extension to a diffeomorphism $\psi:U(K_i)\subset M_i\to U(K_{i'})\subset M_{i'}$ between open neighborhoods.  Therefore,
$$\int_{U(K_i)}\psi^* (j^*_{M_{i'}}\omega )=\int_{U(K_{i'})}(j_{M_{i'}}^*\omega ).$$
In view of the definition of essential points in the special case when $g=\text{id}$, we conclude for all $x\in K_i$ that $\psi (x)=x$ and $T\psi (x)=\text{id}$ so that $\psi^* (j^*_{M_{i'}}\omega)=j^*_{M_i}\omega$ on $K_i$. Hence taking a decreasing sequence  of open neighborhoods of $K_i$ in $M_i$ we obtain, by means of the Lebesgue's convergence theorem, that
$$
\int_{K_i}\omega\vert M_i =\int_{K_{i'}} \omega\vert M_{i'}.
$$
By the same arguments,
$$
\int_{K_i}\omega\vert M_i =\int_{K_{j}} \omega\vert N_j  =\int_{K_{j'}} \omega\vert N_{j'}
$$
for all $i\in I_K$ and $j, j'\in J_K$. If $|x|\in K$, then $\vartheta (|x|)=\sum_{i\in I_K}\sigma_i=\sum_{j\in J_K}\tau_j$ and we can compute,
\begin{equation*}
\begin{split}
\sum_{i\in I}\sigma_i \int_{K_i}\omega\vert M_i&=
\sum_{i\in I_K}\sigma_i \int_{K_i}\omega \vert M_i\\
&= \sum_{j\in J_K}\tau_j \int_{L_j}\omega\vert N_j=\sum_{j\in J}\tau_j \int_{L_j}\omega\vert N_j.
\end{split}
\end{equation*}

This  completes the proof that the number $\muo^U(K)$ does not depend on the choice of a branching structure in $U$.
\end{proof}

Next we investigate the compatibility of our definition of  $\muo^U(K)$ with  restrictions  of the local branching structure $(M_i)_{i\in I}$ and $(\sigma_i)_{i\in I}$ on the open neighborhood $U=U(z_0)$ of the point $z_0\in \sth$.
On $U$ the isotropy group $G=G_{z_0}$ acts by the natural representation $\varphi :G\to \dif (U)$. The construction of a restriction is as follows. For a   point $x\in U\cap \sth$,  we introduce the subgroup $H\subset G$ by
$$H=\{g\in G\vert \, \varphi_g (x)=x\}.$$
From Proposition \ref{prop1.2} one concludes that $g\in H$ if and only if $\Gamma (g, x)$ belongs to the isotropy group $G_x$ of the point $x$. The mapping $\alpha:H\to G_x$ defined by $\alpha (g)=\Gamma (g, x)$ is an isomorphism of groups.  The mapping $G_x\to \dif (U)$, $\gamma \mapsto \varphi_{\alpha^{-1}(\gamma )}$ is a natural representation of the isotropy group $G_x$ on the set $U$. It can be identified with the natural representation $\varphi:H\to \dif (U)$ of the subgroup $H\subset G$.  Now we take an open neighborhood $V=V(x)\subset U$ so that $\varphi_g (V)=V$ if $g\in H$ and $\varphi_g (V)\cap V=\emptyset$ if $g\not \in H$, and set
$$\wt{M}_j=V(x)\cap M_j\quad \text{for $j\in J$},$$
where $J=\{i\in I\vert \, V(x)\cap M_i\neq \emptyset\}$.  Then $(\wt{M}_j)_{j\in J}$ with $(\sigma_j)_{j\in J}$ constitutes  a local branching structure on the open set $V$ on which the isotropy group  $G_x$  (identified with $H\subset G$) acts by the natural representation of $H$.  We shall refer to this construction as the {\bf restriction of the local branching structure} of $U=U(z_0)$ to the neighborhood $V=V(x)\subset U$. Setting
$$M_V=\bigcup_{j\in J}\wt{M}_j=M_U\cap V$$
we consider a measurable subset $K\subset \abs{M_V}\subset S$. As before we define the measure
$\muo^V (K)$ as
$$
\muo^V(K)=\dfrac{1}{\sharp \wt{G}_{\e}}\sum_{j\in J}\sigma_j\int_{\wt{K}_j}\omega\vert \wt{M}_j
$$
where $\wt{K}_j\subset \wt{M}_j$ is the preimage of $K$ under the projection map $\wt{M}_j\to \abs{M_V}$ and
$\wt{G}_{\e}$ is the effective part of the isotropy group $G_x$.

On the other hand, since $\abs{M_V}\subset \abs{M_U}$, we also have $K\subset \abs{M_U}$ and hence  the  measure
$$\muo^U(K)=\dfrac{1}{\sharp G_{\e}}\sum_{i\in I}\sigma_i\int_{K_i}\omega\vert M_i
$$
where $K_i\subset M_i$ is the preimage of $K$ under the projection $M_i\to \abs{M_U}$.

\begin{lem}[{\bf Restrictions}] \label{lemrest} If $V\subset U$ is a restriction of $U$ and $K\in \cl (S)$ is a measurable set contained in $\abs{M_V}$, then
$$\muo^V(K)=\muo^U(K).$$
\end{lem}
\begin{proof}
We identify the isotropy group $G_x$ with the subgroup $H\subset G$. Let
$$G=\bigcup_{r\in R}r\cdot H$$
be the partition of $G$ into the right cosets, where $R\subset G$ is the set of representatives. Then we introduce the open neighborhoods $V_r=\varphi_r (V)$ of the points $\varphi_r (x)$ for all $r\in R$.

\begin{figure}[htbp]
\mbox{}\\[2ex]
\centerline{\relabelbox
\epsfxsize 3.8truein \epsfbox{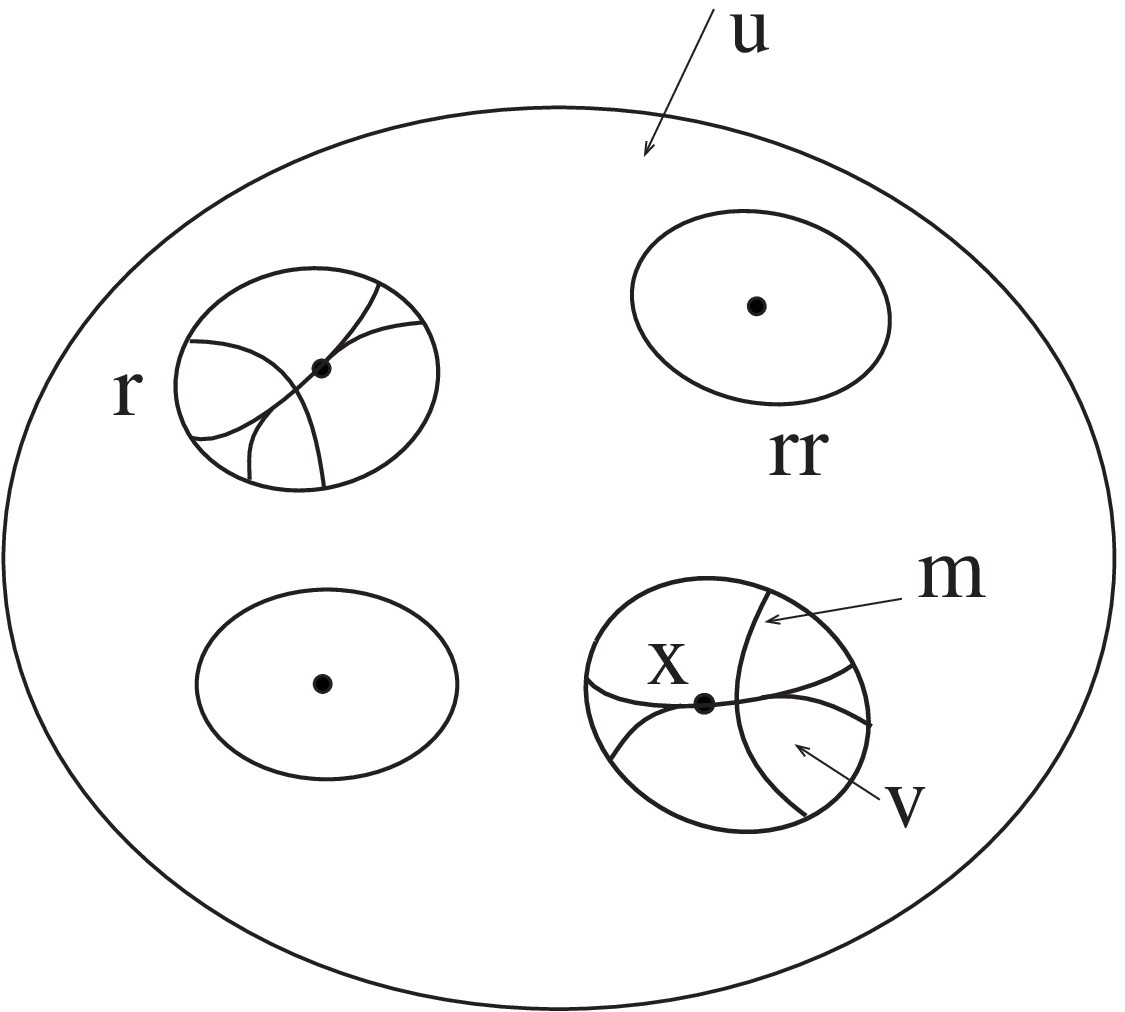}
\relabel {v}{$V$}
\relabel {r}{$V_r$}
\relabel {rr}{$V_{r'}$}
\relabel {u}{$U$}
\relabel {x}{$x$}
\relabel {m}{$\wt{M}_j$}
\endrelabelbox}
\caption{}\label{Fig2}
\mbox{}\\[1ex]
\end{figure}

In $V_r$ we have the natural representation of the isotropy group of $\varphi_r (x)$ by diffeomorphisms from the natural representation of $G$. Abbreviating $\wt{K}=\bigcup_{j\in J}\wt{K}_j\subset V$,  we compute
$$
\muo^U(K)=\dfrac{1}{\sharp G_{\e}}\sum_{i\in I}\sigma_i\int_{K_i}\omega\vert M_i
$$
where $K_i=[\bigcup_{r\in R}\varphi_r (\wt{K})]\cap M_i.$ Hence
\begin{equation*}
\begin{split}
\muo^U(K)&=\dfrac{1}{\sharp G_{\e}}\sum_{i\in I}\sigma_i\sum_{r\in R}\int_{\varphi_r(\wt{K})\cap M_i}\omega\vert M_i\\
&=\dfrac{1}{\sharp G_{\e}}\sum_{r\in R}\sum_{i\in I}\sigma_i\int_{\varphi_r(\wt{K})\cap M_i}\omega\vert M_i.
\end{split}
\end{equation*}
Observe that $\varphi_r (\wt{K})\subset V_r$. Since $\varphi_r$ is an sc-diffeomorphism of $U$ mapping $V$ onto $\varphi_r (V)=V_r$ and $\sth$ is invariant under $\varphi_r$, we obtain another branching structure in $V_r$ consisting of the branches $\varphi_r (\wt{M}_j)_{j\in J}$ and the weights $(\sigma_j)_{j\in J}$. Hence, in view of
Lemma \ref{lemA} and since $\varphi_g^*\omega=\omega$ for all $g\in G$,
\begin{equation*}
\begin{split}
\sum_{i\in I}\sigma_i\int_{\varphi_r(\wt{K})\cap M_i}\omega\vert M_i&=\sum_{j\in j}\sigma_j\int_{\varphi_r(\wt{K}_j)}\omega\vert  \varphi_r(\wt{M}_j)\\
&=\sum_{j\in J}\sigma_j\int_{\wt{K}_j}\varphi_r^*\omega\vert  \wt{M}_j=
\sum_{j\in J}\sigma_j\int_{\wt{K}_j}\omega\vert  \wt{M}_j.
\end{split}
\end{equation*}
Consequently,
$$\muo^U(K)=\dfrac{1}{\sharp G_{e}}\abs{R}\sum_{j\in J}\sigma_j\int_{\wt{K}_j}\omega\vert \wt{M}_j.$$
Now, $\abs{R}=\frac{\abs{G}}{\abs{H}}$. Since the groups $H$ and $G_x$ are isomorphic, we have $\abs{H}=\abs{G_x}$. Moreover, $\sharp G_{\e}=\frac{\abs{G}}{\abs{G_0}}$ and $\sharp \wt{G}_{\e}=\frac{\abs{G_x}}{\abs{(G_x)_0}}.$
In view of Proposition 5.6 in \cite{HWZ3}, the ineffective parts $(G_x)_0$ and $G_0$ are isomorphic as groups so that
$\abs{(G_x)_0}=\abs{G_0}$. Summarizing,
$$
\dfrac{1}{\sharp G_{\e}}\cdot \abs{R}=\dfrac{1}{\dfrac{\abs{G}}{\abs{(G_x)_0}}} \cdot \dfrac{\abs{G}}{\abs{G_x}}=\dfrac{1}{\sharp \wt{G}_{\e}}.
$$
We have verified that $\muo^U(K)=\muo^V(K)$ as claimed.
\end{proof}
Next we establish a property which we call morphism
invariance. Having a subset $K$ of $S$ and a local branching
structure on the open  set $U\subset X$ so that $|M_U|=K$,  it is possible that there
is a second open set $U'\subset X$ which is disjoint from $U$, but also
equipped with a local branching structure so that $|M_{U'}|=K$. Of
course, we expect the associated measures constructed from $\omega$
to be the same. This follows from the next lemma.
\begin{lem}[{\bf Morphism Invariance}] \label{lemD} Let  $U=U(x)$ and $U'=U(x')$ be open subsets of the ep-groupoid $X$ invariant under the natural representations  of
$G_{x}$ and $G_{x'}$, respectively. Assume $\varphi:x\rightarrow x'$ is a morphism
 so that the corresponding
$t\circ s^{-1}$ gives an sc-diffeomorphism $U\rightarrow U'$. Moreover, assume that
$U$ and $U'$ contain local branching structures  of  the $n$-dimensional
branched ep-subgroupoid $\Theta$.   If  $\omega$ is an sc-differential $n$-form on $X_\infty$,  then the  associated measures $\mu_\omega^U$ and $\mu_\omega^{U'}$ coincide.\end{lem}
\begin{proof}
We  already  know from Lemma \ref{lemB} that the measures  $\mu_\omega^U$ and
$\mu_\omega^{U'}$ do not depend on the local branching structures.
We fix a local branching structure $(M_i)_{i\in I}$ with the weights $(\sigma_i)_{i\in I}$ on the sets $U$. We  push  the local branches $M_i$ forward by an  sc-diffeomorphism $t\circ s^{-1}$  to obtain local branches $M_i'$ on $U'$. Then $(M_i')_{i\in I}$   with the same weights $(\sigma_i)_{i\in I}$ is a local branching structure in $U'$.  Since
$\omega$ is invariant under this sc-diffeomorphism,  the result follows.
\end{proof}
We continue with the proof of Theorem \ref{th1}. For every
point $s\in S$,  we take  $x\in \supp \Theta$ so that  $|x|=s$, and an
open subset $U(x)\subset X$ invariant under the natural representation of
$G_x$ and  containing a  local branching structure. Since $S$ is compact, we find
finitely many points $x_1, \ldots  ,x_k$ so that the open sets $\abs{U(x_i)}$ cover $S$.  We abbreviate $U_i=U(x_i)$ and denote by $U_i^\ast =\pi^{-1}(\pi (U_i))$
the saturation of $U_i$. We add another open set $U^\ast_0$ so that
${(U^\ast_i)}_{1\leq i\leq k}$ is a saturated open cover of $X$ and,  in
addition,  $(\abs{U^\ast_i\setminus \overline{U^\ast_0}})_{1\leq i\leq k}$,  is an open  cover of $S$.

Then  there exists an sc-smooth partition of unity, Theorem \ref{scpounity},
 $(\beta_i)_{0\leq i\leq k}$ on $X$  such that $\supp \beta_i \subset U_i^\ast$  and each  $\beta_i$  is compatible with
morphisms. Moreover,
$$
\sum_{i=1}^k \beta_i =1 \ \ \hbox{on}\ \bigcup_{i=1}^k M_{U_i}.
$$
Fix an sc-differential $n$-form $\omega$ and define $\omega_i=\beta_i\omega$. Then  $\supp \omega_i    \subset U_i^\ast$,
$$
\omega = \sum_{i=0}^k \omega_i,
$$
and $\omega_0=0$ on $\bigcup_{i=1}^k M_{U_i}$.   If $b_i$, $i=1,\ldots ,k$, is the induced continuous partition of unity on $S$ defined by $b_i(|x|)=\beta_i (x)$ for $|x|\in S$, then the support $\supp b_i\subset \abs{M_{U_i}}$ is compact.  Hence, for every $i=1,\ldots  ,k$,  we can construct
a measure on ${\mathcal L}(S)$ by
$$
K\mapsto \mu_{\omega_i}^{U_i}(K\cap \abs{M_{U_i}})=\mu_{\omega_i}^{U_i}(K\cap \supp b_i).
$$
By abuse of notation we  denote it by $\mu_{\omega_i}^{U_i}$.  Then  we
define the measure $\mu_\omega$ on ${\mathcal L}(S)$ by
\begin{equation}\label{deff}
\mu_\omega := \sum_{i=1}^k \mu_{\omega_i}^{U_i}.
\end{equation}
Next we shall  demonstrate  that the measure $\mu_\omega$ is well defined and
independent of the choices made in the construction on the
right-hand side. Once the lemma is proved, Theorem \ref{th1} follows immediately.
%\textcolor{red}

\begin{lem}\label{lemma3.8a}
 For every  $s\in S$,  there exist an open neighborhood $O(s)$ and an open set
$U=U(x)\subset X$ covering $O(s)$ such that the following holds. The set $U$ is  invariant under the natural representation  of the isotropy group
$G_{x}$, contains a local branching structure, and  for every  measurable subset $K\in {\mathcal L}(S)$ contained in a compact subset
of $O(s)$,
$$
\mu_\omega(K) = \mu_\omega^U(K).
$$
\end{lem}
In  other words, the formula  \eqref{deff}  defines a signed measure $\mu_\omega$ on
$(S,{\mathcal L}(S))$ which locally coincides with the previously
constructed local measures $\mu_\omega^U$. There can be, of course,
at most one such measure. Hence we conclude  the existence and the  uniqueness of
a measure $\mu_\omega$ on $(S,{\mathcal L}(S))$ which has the
properties listed in Theorem \ref{th1}.
\begin{proof}
Assume that $s\in S$ is given. Denote by $I_s\subset \{1,\ldots ,k\}$ the set
of all $i$ such that there exists $x_i\in U_i$ with $[x_i]=s$.  Now we choose a small open neighborhoods $V_i\subset U_i$ of $x_i$ on which the isotropy  group $G_{x_i}$ acts by its natural representation. Since all points $x_i$ belong to the same equivalence class $s\in S$, we can  choose these open neighborhoods $V_i$ such that they are, in addition, sc-diffeomorphic by means of diffeomorphisms  of the form $t\circ s^{-1}$. We define the open neighborhood $O(s)\subset S$ by $O(s)=S\cap \abs{M_{V_i}}$ for some and hence for every $i\in \{1, \ldots, k\}$. In every open neighborhood $V_i$ we have the induced branching structure from the branching structure of $U_i$, namely the restriction of the local branching structure of $U_i$ to $V_i$. If $K\in {\mathcal L}(S)$ is a  measurable set contained in a compact subset of the open neighborhood $O(s)$ we compute

\begin{equation*}
\begin{split}
\mu_\omega(K)&=\sum_{I_s} \mu_{\omega_i}^{U_i}(K)\ \ \text{(Definition)}\\
&= \sum_{I_s} \mu_{\omega_i}^{V_i}(K)\ \ \text{(Restrictions)}\\
&=\sum_{I_s}\mu_{\omega_i}^{V_1}(K)\ \text{(Morphism Invariance)}\\
&=\sum_{I_s} \dfrac{1}{\sharp G^1_{\e} } \sum_J \sigma_j\int_{K_j}\omega_i\vert M_j\ \ \text{(Independence)}\\
&= \dfrac{1}{\sharp G^1_{\e} }\sum_J\sigma_j \int_{K_j}\sum_{I_s}\omega_i\vert M_j\\
&=\dfrac{1}{\sharp G^1_{\e} }\sum_J\sigma_j\int_{K_j} \omega|M_j\ \text{(Partition of unity)}\\
&=\mu^{V_1}_{\omega}(K).
\end{split}
\end{equation*}
Putting $U=V_1$ completes the proof of Lemma \ref{lemma3.8a} and hence of Theorem
\ref{th1}.
\end{proof}

\subsection{Stokes' Theorem}
In this section we prove Theorem 1.10.  Since we already know that the definitions  of the measures do not depend on the choices involved,  the proof of Stokes'
theorem is  rather straightforward.

As before we assume that the ep-groupoid is build on
separable sc-Hilbert spaces.  Let  $\Theta$ be  an oriented
branched ep-subgroupoid whose support  $\supp \Theta$ has a compact orbit space $S=\abs{\supp \Theta}$.   We may assume that $\Theta$ is of dimension $n$ and
$\omega$  is an  sc-differential form of degree $n$ on $X_\infty$.

By the compactness of $S$, we find finitely many  open sets $U_k=U(x_k)$, $1\leq  k\leq n$,  in $X$ such that  the sets $\abs{U_k}$ cover $S$ and every $U_{k}$ is invariant with respect to  the natural $G^k$-action where $G^k=G_{x_k}$ is the  isotropy group of  the point $x_k$.
Moreover,  every $U_k$ contains  the local  branching structures  $(M^k_i)_{i\in I^k}$  with the associated weights $(\sigma_i^k)_{i\in I^k}$. The local branches $M_i^k$  are properly embedded finite dimensional submanifolds of $X$, all  of the same dimension $n$.  Abbreviating by $U_{k}^*$ the saturations of the sets $U_{k}$, we add  another open set $U^\ast_0$ so that
the sets $U^\ast_{k}$  cover $X$ and
the sets $\abs{U^\ast_{k}\setminus \overline{U^\ast_0}}$  still cover $S$.
Next we  take an sc-smooth partition of unity  $\beta_0,\ldots ,\beta_n$ on $X$  such that  $\supp \beta_k\subset U^{\ast}_k$ and  each  $\beta_k$  is compatible with
morphisms.  Then
$$\text{$\sum_{k=1}^n \beta_{k} =1$\,  on  \, $\bigcup_{k=1}^n  M_{U_{k}}$.}$$
Setting  $\omega_k=\beta_k \omega$, we have $\supp \omega_k\subset U^*_k$ and
$\omega=\sum_{k=1}^n\omega_k$.

If  the set $U_k$ does not contain parts of $\partial S$,  we
conclude  by the standard Stokes' theorem
$$
\int_{M_i^k} d\omega_{k} =0.
$$
In the case of  boundary components we obtain,  again by Stokes' theorem,
$$
\int_{M_i^k} d\omega_k=\int_{\partial M_i^k} \omega_{k}.
$$
Summing  up  the left hand side as well as the right hand side we obtain
\begin{equation*}
\begin{split}
\mu_{d\omega}^{(S,\theta)}(S) &=
\sum_{k=1}^n \dfrac{1}{\sharp G_{\e}^k} \sum_{i\in I^k} \sigma^k_i \int_{M_i^k} d\omega_k\\
&=\sum_{k=1}^n\dfrac{1}{\sharp G_{\e}^k}
\sum_{i\in I^k}\sigma^k_i \int_{\partial M_i^k} \omega_k\\
&=\mu_{\omega}^{(\partial S,\theta)}(\partial S).
\end{split}
\end{equation*}
This completes the proof of Stokes' theorem in the branched context.

\subsection{Proofs of Theorem \ref{th4} and Theorem \ref{thmst1}}
Having proved Stokes' theorem in the  ep-groupoids  context, we are ready to prove Theorem \ref{th4} and Theorem \ref{thmst1}. For the convenience we restate these theorems below.
\begin{thnnumber}[{\bf Equivalences}]%\label{th4}
Let $F:X\to  Y$ be an equivalence between
ep-groupoids $X$ and $Y$.  Assume that  $\Theta:Y\rightarrow {\mathbb Q}^+$ is  an
oriented $n$-dimensional branched ep-subgroupoid of $Y$  whose orbit
space $S=\abs{\supp \Theta}$ is compact and equipped with the weight
function $\vartheta:S\to \Q^+$ defined by $\vartheta (|y|)=\Theta
(y)$ for $|y|\in S$. Define the $n$-dimensional branched
ep-subgroupoid  on $X$ by $\Theta':=\Theta\circ F:X\to \Q^+$ and denote by $S'$ and $\vartheta'$  the associated orbit space   and the  weight function on $S'$. Moreover, assume that
$\Theta'$ is equipped with the induced orientation. Then, for every
sc-differential $n$-form $\omega$ on $Y_\infty$,
$$
\mu_{\omega}^{(S,\vartheta)}\circ \abs{F} = \mu_{\omega'}^{(S',\vartheta')}
$$
where the $n$-form $\omega'$  is the pull back form $\omega'=F^*\omega$
on $X_\infty$.
Similarly,
$$
\mu_{\tau}^{(\partial S,\vartheta)}\circ \abs{F}=
\mu_{\tau'}^{(\partial S',\vartheta')}
$$
for every $(n-1)$-form $\tau$ on $Y$.
\end{thnnumber}

\begin{proof}
It suffices to prove that $\mu_{\omega}^{S', \vartheta'}(K')=\mu_{\omega}^{S, \vartheta}\circ \abs{F}(K')$ for sets $K'\in {\mathcal L}(S')$ contained in open sets which are invariant under the natural action and contain local branching structure.
Take  $y\in \supp \Theta$ and let $U=U(y)$ be an open neighborhood of $y$  invariant under the  natural action of the isotropy group $G_y$  and containing  a local branching structure  $(M_i)_{i\in I}$  with the weights $(\sigma_i)_{i\in I}$. Of course, the point  $y$ does not need to be in the set $F(X)$. However, using the fact that $\abs{F}:\abs{X}\to \abs{Y}$ is a homeomorphism, we can find
a point  $x\in X$ and a morphism $\varphi:F(x)\rightarrow y$.   This morphism extends to a local sc-diffeomorphims of the form $t\circ s^{-1}$. Replacing  if necessary $U(y)$ by a smaller open set, taking a suitable open neighborhoods $U'=U(x)$, and composing $F$ with $t\circ s^{-1}$, we may assume that $y=F(x)$ and that  the map $F:U'\to U$ is an sc-diffeomorphism.
Then  $\sharp G_x=\sharp G_y$ and  $\varphi_{F(g)}=F\circ \varphi_{g}\circ F^{-1}$. Hence $U'$ is an $G_x$-invariant neighborhood of $x$.  Putting   $M_i'=F^{-1}(M_i)$, we  see that   $(M_i')_{i\in I}$ together with weights $(\sigma)_{i\in I}$ is a local branching structure in $U'$.  If $K'\in {\mathcal L}(S')$ is  a compact subset of $\abs{\supp \Theta'\cap U'}$, then $K=\abs{F}(K')$  so that $K\in {\mathcal L}(S)$.  Since  $\omega'=F^*\omega$, we find that
\begin{equation*}
\begin{split}
\mu_{\omega}^{(S,  \vartheta )}\circ \abs{F}(K')&=\mu_{\omega}^{(S,  \vartheta )}(K)
=\dfrac{1}{\#G_y}\sum_{i\in I}\sigma_i \int \omega\vert M_i\\
&=\dfrac{1}{\#G_x}\sum_{i\in I}\sigma_i \int \omega'\vert M_i'
=\mu_{\omega'}^{(S', \vartheta )}(K).
\end{split}
\end{equation*}
This completes the proof of Theorem \ref{th4}.
\end{proof}
Next we restate and  prove Theorem \ref{thmst1}.
\begin{thnnumber}%\label{thmst1}
Let $Z$ be a polyfold and $S\subset Z$ be an oriented compact branched suborbifold
defined by the equivalence class $[(X, \beta ,\Theta)]$ and equipped with the weight function $w:S\to \Q^+\cap (0,\infty )$.
For an sc-differential $n$-form $\tau$ on  $Z_\infty$ and $K\in {\mathcal L}(S)$,  define
$$
\int_{(K,w)}\tau:=\int_{\beta^{-1}(K)}
d\mu_\omega^{(\beta^{-1}(S),\vartheta)}=\mu_\omega^{(\beta^{-1}(S),\vartheta)}(\beta^{-1}(K)),
$$
where  the equivalence class $\tau$ is represented by the triple $(X, \beta ,\omega)$ and the weight function $\vartheta$  on
$\beta^{-1}(S)=\abs{\supp \Theta}$ is defined by $\vartheta (|x|)=\Theta (x).$
Then   the integral  $\int_{(K,w)}\tau$ is
well defined. Moreover,   if $\tau$ is an sc-differential $(n-1)$-form on $Z_\infty$,  then
$$
\int_{(\partial S,w)}\tau =\int_{(S,w)} d\tau.
$$
\end{thnnumber}

\begin{proof}
Let $(X, \beta, \Theta)$ and $(X', \beta', \Theta')$ be two equivalent representatives of the oriented compact branched suborbifold $S\subset Z$.  This means that there exist a third polyfold structure $(X'', \beta'')$ and two quivalences $X\xleftarrow{F} X''\xrightarrow{F'} X'$ satisfying $\beta''=\beta \circ \abs{F}=\beta'\circ \abs{F'}$. Then we have $\Theta'':=\Theta'\circ F'=\Theta \circ F$.  Moreover, we require that the equivalences $F$ and $F'$ preserve orientations.  Now if
$\tau$ is  an sc-dffferential form on the polyfold $Z_\infty$, let $\omega$ and $\omega'$ be sc-differential $n$ forms on $X_\infty$ and $X'_\infty$ respectively, so that $F^*\omega=(F')^*\omega'$. Hence the triples $(X, \beta, \omega)$ and $(X', \beta',\omega')$ belong to the equivalence class $\tau$.

Assume that $K\subset {\mathcal L}(S)$. Then Theorem \ref{th4} applied to the equivalence $F:X''\to X$,  oriented $n$-dimensional branched ep-subgroupoids $\Theta:X\to \Q^+$ and  $\Theta'':X''\to \Q^+$, and
sc-differential forms $\omega$ on $X_\infty$ and $F^*\omega$ on $X''_\infty$  gives
\begin{equation*}
\begin{split}
\mu_{\omega}^{(\beta^{-1}(S), \vartheta)}(\beta^{-1}(K))&=
\mu_{F^*\omega}^{((\beta'')^{-1}(S), \vartheta'')}( \abs{F}^{-1}\circ \beta^{-1}(K))\\
&=
\mu_{F^*\omega}^{((\beta'')^{-1}(S), \vartheta'')}( (\beta'')^{-1}(K)).
\end{split}
\end{equation*}
Similarly,  Theorem \ref{th4} applied to the equivalence $F:X''\to X'$,  oriented $n$-dimensional branched ep-subgroupoids $\Theta':X\to \Q^+$ and  $\Theta'':X''\to \Q^+$, and
sc-differential forms  $\omega'$ on $X_{\infty}$ and $(F')^*\omega'$ on $X_{\infty}''$ gives
$$\mu_{\omega'}^{((\beta')^{-1}(S), \vartheta')}((\beta')^{-1}(K))=\mu_{(F')^*\omega'}^{((\beta'')^{-1}(S), \vartheta'')}( (\beta'')^{-1}(K)).$$
Since $F^*\omega=(F')^*\omega'$,  we obtain
$$\mu_{F^*\omega}^{(\beta'')^{-1}(S), \vartheta'')}( (\beta'')^{-1}(K))=\mu_{(F')^*\omega'}^{(\beta'')^{-1}(S), \vartheta'')}( (\beta'')^{-1}(K))$$
and  conclude that
$$\mu_{\omega}^{(\beta^{-1}(S), \vartheta)}(\beta^{-1}(K))=\mu_{\omega'}^{((\beta')^{-1}(S), \vartheta')}((\beta')^{-1}(K)).$$
Consequently,   the integral $\int_{(K, w)}\tau$ is well defined.

The remaining part of the theorem is a consequence of Theorem \ref{thmst0}.
Choose  a  representative  $(X, \beta , \Theta )$  of the oriented $n$-dimensional  branched  suborbifold $S$  equipped with the weighting function $w:S\to  \Q^+\cap (0,\infty )$. Next we choose an sc-differential $(n-1)$-form  $\omega$  such that  $(X, \beta, \omega )$ is  a  representative of the differential form $\tau$ on $Z_\infty$. Since $d[\omega]=[d\omega]$,  the triple $(X, \beta, d\omega)$ is a representative of the  sc-differential $n$-form $d\tau$ on $Z_\infty$.
Putting $K=S$, we find
\begin{equation*}
\begin{split}
\int_{(\partial S, w)}d\tau&=\mu_{d\omega}(\beta^{-1}(S), \vartheta )(\beta^{-1}(S))\\
&=\mu_{\omega}^{\partial \beta^{-1}(S), \vartheta)}(\partial \beta^{-1}(S))\\
&=
\mu_{\omega}^{ \beta^{-1}(\partial S), \vartheta)}( \beta^{-1}(\partial S))=\int_{(\partial S, w)}\tau.
\end{split}
\end{equation*}
The proof of the theorem is completed.
\end{proof}

\section{Appendix}

\subsection{Finite dimensional submanifolds}\label{fdimsubm}
In \cite{HWZ1}  we  introduced  the concept  of a strong finite dimensional  submanifold of an
M-polyfold. It  carries  the structure of a manifold in a natural way. In \cite{HWZ2} we introduced the more general notion of a finite dimensional submanifold which we recall below.    Submanifolds according to the new definition  have natural manifold structures. Moreover,  strong
finite dimensional submanifolds  are also submanifolds according to the new definition.
The manifold structures induced in both cases are the same.

We recall from  Definition 3.1  in \cite{HWZ1} that the splicing
${\mathcal S}=(\pi, E, V)$  consists of the sc-Banach space $E$, the open subset $V$ of a partial quadrant $C$ in the sc-Banach space $W$,  and the sc-smooth map $\pi:V\oplus E\to E$ has the property that for every $v\in V$ the mapping
$$\pi_v=\pi (v, \cdot ):E\to E$$
is a  bounded linear projection operator. The splicing core $K^{\mathcal S}$  associated with the splicing
${\mathcal S}=(\pi, E, V)$ is the set
$$K^{\mathcal S}=\{(v, e)\in V\oplus E\vert \, \pi_ve=e\}.$$
\begin{defn}\label{sub1}
Let $X$ be an M-polyfold and  let $M\subset X$ be a subset  equipped with the
induced topology. The subset  $M$ is called a  {\bf finite dimensional submanifold} of $X$
provided the following holds.
\begin{itemize}
\item[$\bullet$]  The subset $M$ lies in $X_\infty$.
\item[$\bullet$] At  every point $m\in M$  there exists an  M-polyfold chart
$$
(U, \varphi, (\pi,E,V))
$$
where  $m\in U\subset X$ and $\varphi:U\rightarrow
O$ is a homeomorphism satisfying $\varphi (m)=0$ onto the  open neighborhood $O$ of $0$ in
the splicing core $K^{\mathcal S}$ associated with the splicing  ${\mathcal S}=(\pi, E, V)$
and having the following property.  There
exist  a finite-dimensional smooth linear subspace $N\subset W\oplus
E$ in good position to $C$, a  corresponding sc-complement
$N^\perp$, and an open neighborhood $Q$ of $0\in C\cap N$,  and an
$\ssc$-smooth map $A:Q\rightarrow N^\perp$ satisfying  $A(0)=0$, $DA(0)=0$
so that the map
$$
\Gamma:Q\rightarrow W\oplus E:q\rightarrow q+A(q)
$$
has its image in $O$ and the image of the composition
$\Phi:=\varphi^{-1}\circ\Gamma:Q\rightarrow U$ is equal to  $M\cap U$.
\item[$\bullet$] The map $\Phi:Q\rightarrow M\cap U$ is a homeomorphism.
\end{itemize}
The map $\Phi:Q\rightarrow U$ is called a  {\bf good parametrization}
of a neighborhood of $m$   in $M$.
\end{defn}
In  other words,   a  subset $M$  of an  M-polyfold $X$  consisting of smooth points
is a submanifold if for every $m\in M$ there is  a good
parametrization of an open neighborhood of $m$ in $M$. It has been proved in \cite{HWZ2}  that the transition maps  $\Phi\circ \Psi^{-1}$ defined  by two good
parameterizations $\Phi$ and $\Psi$  are smooth, so the inverses of the good parametrizations define an atlas of smoothly compatible charts.  Consequently, a finite dimensional submanifold is in a natural way a manifold with boundary with corners. By construction the map $\Phi:Q\rightarrow M$ occurring in  Definition \ref{sub1} is sc-smooth.
\subsection{The Lie bracket [A, B]}\label{poly}
In this appendix we shall define the Lie bracket $[A, B]$ of two sc-smooth vector fields on an M-polyfold $X$. A vector field is an sc-smooth section of the tangent bundle $TX\to X^1$. In order to define the Lie derivative,  we go into a local chart $O\subset K^{\mathcal S}$ of $X$. Here  $O$ is an open subset of the splicing core
$K^{\mathcal S}=\{(v, e)\in V\oplus E\vert \, \pi_ve=e\}$ associated with  the splicing
${\mathcal S}=(\pi, E, V)$. The triple consists of the sc-Banach space $E$, the open subset $V$ of a partial quadrant $C$ in the sc-Banach space $W$, and the sc-smooth map $\pi:V\oplus E\to E$ which has the property that for every $v\in V$ the mapping
$$\pi_v=\pi (v, \cdot ):E\to E$$
is a  bounded linear projection operator. The fiber $T_{(v, e)}O$ of the tangent bundle $TO\to O^1$ over the point $(v, e)\in O^1$ is given by
\begin{equation}\label{eqtan}
T_{(v, e)}O=\{(\delta v, \delta e)\in W\oplus E\vert \, T\pi ((v, e), (\delta v, \delta e))=(e, \delta e)\}.
\end{equation}

We recall that the tangent map $T\pi$ of the map $\pi$ at $(v, e)\in O^1$ is defined by
$ T\pi ((v, e), (\delta v, \delta e))=(\pi (v, e), D\pi (v, e)(\delta v, \delta e))\in E^1\oplus E$. Therefore, we deduce from \eqref{eqtan} that $\pi (v, e)=e$ and
\begin{equation*}
\begin{split}
\delta e&=D\pi (v, e)(\delta v, \delta e)\\
&=D_v\pi (v, e)\cdot \delta v+\pi (v, \delta e),
\end{split}
\end{equation*}
if $(\delta v, \delta e)\in T_{(v, e)}O$. Consequently, the principal part of a vector field is in our local coordinates at the points $(v, e)\in O^1$ represented by
$$A(v, e)=(\alpha (v, e), \beta (v, e))\in W_0\oplus E_0$$
where the two maps $\alpha$ and $\beta$ satisfy the identity
\begin{equation}\label{eqtan1}
\beta (v, e)=\pi (v, \beta (v, e))+[D_v\pi (v, e)]\cdot \alpha (v, e).
\end{equation}
Similarly, the second vector field $B$ has at points $(v, e)\in O^1$ the representation
$$A(v, e)=(\gamma (v, e), \delta (v, e))\in W_0\oplus E_0$$
with $\gamma$ and $\delta$ satisfying the equation
\begin{equation}\label{eqtan2}
\delta  (v, e)=\pi (v, \beta (v, e))+[D_v\pi (v, e)]\cdot \gamma  (v, e).
\end{equation}
At a point $(v, e)\in O_2$,  we can differentiate the identity \eqref{eqtan1} in the direction of the vector field $B(v, e)\in W_1\oplus E_1$ and obtain
\begin{equation*}
\begin{split}
D\beta (v, e)\cdot B(v, e)&=\pi (v,  D\beta (v, e)\cdot B(v, e))\\
&\phantom{=}+[D_v\pi (v, e)]\cdot D\alpha (v, e))\cdot B(v, e)\\
&\phantom{=}+[D_v\pi (v, \beta (v, e))]\cdot  \gamma (v, e)\\
&\phantom{=}+[D_v\pi (v,\delta (v, e))]\cdot \alpha (v, e)\\
&\phantom{=}+[D_v^2\pi(v, e)] (\gamma (v, e), \alpha (v, e)).
\end{split}
\end{equation*}
Similarly,  differentiating \eqref{eqtan2} at the point $(v, e)\in O_2$  in the direction of the vector $A(v, e)$,  one finds
\begin{equation*}
\begin{split}
D\delta (v, e)\cdot A(v, e)&=\pi (v,  D\delta(v, e)\cdot A(v, e))\\
&\phantom{=}+[D_v\pi (v, e)]\cdot D\gamma (v, e))\cdot A(v, e)\\
&\phantom{=}+[D_v\pi (v, \delta (v, e))]\cdot  \alpha (v, e)\\
&\phantom{=}+[D_v\pi (v,\beta(v, e))]\cdot \gamma (v, e)\\
&\phantom{=}+[D_v^2\pi(v, e)] (\alpha (v, e), \gamma (v, e)).
\end{split}
\end{equation*}
Taking  the difference  and using the symmetry of the second derivative
$$[D_v^2\pi(v, e)] (\gamma (v, e), \alpha (v, e))
=[D_v^2\pi(v, e)] (\alpha (v, e), \gamma (v, e)),$$
one finds
\begin{equation*}
\begin{split}
D\beta (v, e)\cdot B(v, e)&-D\delta (v, e)\cdot A(v, e)\\
&=\pi (v,  D\beta (v, e)\cdot B(v, e)-D\delta (v, e)\cdot A(v, e))\\
&\phantom{=}+D_v\pi (v, e)\cdot [  D\alpha (v, e)\cdot B(v, e)-D\gamma (v, e)\cdot A(v, e)].
\end{split}
\end{equation*}
This shows that
$$ (D\alpha (v, e)\cdot B(v, e)-D\gamma (v, e)\cdot A(v, e),  D\beta (v, e)\cdot B(v, e)-D\delta (v, e)\cdot A(v, e))$$
belongs to the fiber  $(T_{(v, e)}O)_1$ whenever $(v, e)\in O_2$.

Therefore, we can define the {\bf Lie bracket } of the  vector fields $A$ and $B$ at the point $(v, e)\in O_2$ by
$$[A, B](v, e)=DA(v, e)\cdot B(v, e)-DB(v, e)\cdot A(v, e).$$
More explicitly,
\begin{equation*}
[A, B](v, e)=
(D\alpha \cdot B-D\gamma \cdot A,  D\beta \cdot B-D\delta \cdot A
\end{equation*}
where all the terms are evaluated at $(v, e)\in O_2$.  Observe that if $(v, e)\in O_2$,  then $[A, B](v, e)$ belongs to the  level $1$.  That the definition of $[A, B]$ does not depend on the choice of local coordinates is a consequence of the next  lemma.
The lemma  also implies that for an ep-groupoid $X$ the Lie bracket is compatible with the morphisms.
\begin{lem}\label{proplie2}
Let  $\varphi:X\to X$ be  a local sc-diffeomorphism. If $A$ and $B$ are two sc-smooth vector fields on $X$ satisfying  $T\varphi (A(x))=A(\varphi (x))$ and $T\varphi (B(x))=B(\varphi (x))$ at points in $X_2$, then also
$$T\varphi ([A, B](x))= [A, B](\varphi (x))$$
at $x\in X_2$.
\end{lem}
\begin{proof}
Working in local coordinates, we can assume that $x$ varies in an open neighborhood $O$ in the splicing core $K\subset V\oplus E$ and $\varphi (x)$ in $O'$ of the splicing core $K'\subset V'\oplus E'$. Denoting the sc-smooth vector fields in local coordinates again by $A$ and $B$ we have, by assumption,  that $D\varphi \cdot A=A\circ \varphi$ and
$D\varphi \cdot B=B\circ \varphi$. Differentiation of these relations at points on the  level $2$  gives
\begin{align*}
D(A\circ \varphi  )\cdot B&=D^2\varphi\cdot (A, B)+D\varphi \cdot [DA\cdot B]\\
D(B\circ \varphi  )\cdot A&=D^2\varphi\cdot (B, A)+D\varphi \cdot [DB\cdot A].
\end{align*}
Taking the difference, using the symmetry of the second derivative and the chain rule, one computes
\begin{equation*}
\begin{split}
D\varphi \cdot [A, B]&=D\varphi \cdot [DA\cdot B-DB\cdot A]\\
&=D (A\circ \varphi)\cdot B-D (B\circ \varphi)\cdot A\\
&=D A\circ \varphi \cdot [D\varphi \cdot B]-DB\circ \varphi\cdot [D\varphi \cdot A]\\
&=D A\circ \varphi\cdot [B\circ \varphi ]-DB\circ \varphi\cdot [A\circ \varphi ]\\
&=[A, B]\circ \varphi.
\end{split}
\end{equation*}
The proof of  Proposition \ref{proplie2} is complete.
\end{proof}
The next  proposition summarizes  our discussion.
\begin{prop}
If $X$ is an M-polyfold and $A$ and $B$ are sc-smooth vector fields on $X$, then $[A, B]$ is an sc-smooth vector field on $X^1$,that is, $[A, B]$ is a section of $T(X^1)\to X^2$.
\end{prop}

\subsection{Sc-Smooth Partitions of Unity}\label{scpartition}
In  this section  we  prove the existence of an sc-smooth partition of unity on an ep-groupoid.  We consider an ep-groupoid  whose sc-structure is  based on separable sc-Hilbert spaces.
We view $[0,1]$ as a category with only the identity morphisms.   An sc-smooth functor $f:X\to [0,1]$ on $X$  is an sc-smooth map on the object M-polyfold  which  is invariant under morphisms, that is, $f(x)=f(y)$ if there exists a morphism $h:x\to y$.

\begin{defn}  Let $X$ be an ep-groupoid and let ${\mathcal U}=(U_{\alpha})_{\alpha \in A}$ be an open cover  of $X$ consisting of saturated sets.  An sc-partition of unity  $(g_\alpha)_{\alpha \in A}$ subordinate to ${\mathcal U}$ consists of  the  locally finite  collection of sc-smooth functors $g_\alpha :X\to [0,1]$ so that $\sum_{\alpha \in A}g_\alpha =1$ and $\supp g_{\alpha}\subset U_{\alpha}$ for every $\alpha \in A$.
\end{defn}
The existence of an sc-smooth partition of unity depends on a sufficient supply of sc-smooth functions.  }  We note the following result proved in \cite{Fathi} for separable Hilbert spaces.

\begin{lem}\label{l1}
Let $U$ and $W$ be  open subsets of a separable Hilbert space $H$  such that $\ov{W}\subset  U$. Then there exists a smooth function $f:H\to [0,1]$ having support contained in $U$ and $f=1$ on $\ov{W}$.
\end{lem}
The proof  of  Lemma \ref{l1}  extends  easily to the case when  $U$ and $W$ are  open subsets of  a partial quadrant in  a separable Hilbert space. In the next lemma we extend  Lemma \ref{l1}  to the sc-context.
\begin{lem}\label{l2}
Let $W$ and $U$ be open subsets of a splicing core $K$ such that $\ov{W}\subset U$. Then there exists an sc-smooth function $f: K\to [0,1]$ such that $f$ has its support in $U$ and is equal to $1$ on $W$.
\end{lem}
\begin{proof}
Assume that  $K=K^{{\mathcal S}}=\{(v, e)\in V\oplus E\vert \, \pi_v (e)=e\}$ is the splicing core associated with the splicing ${\mathcal S}=(\pi, E, V)$.  Here $V$ is an open subset of a partial cone $C$ in a separable sc-Hilbert space $Z$,  $E$ is a separable sc-Hilbert space, and $\pi:V\oplus E\to E$ is an sc-smooth map such that $\pi (v, \cdot ):=\pi_v:E\to E$ is a bounded  linear projection for every $v\in V$.  Consider  $\Phi:W\oplus E\to W\oplus E$ defined by $\Phi (v, e)=(v, \pi (v, e))$. The map
$\Phi$ is sc-smooth and, in particular, continuous from level $0$ to level $0$ of $V\oplus E$.
Moreover, $\Phi (V\oplus E)=K$. Put  $W'=\Phi^{-1}(W)$ and $U'=\Phi^{-1}(U)$.
Then $W'$ and $U'$ are open and since  $\Phi^{-1}(\ov{W})$ is closed,  we get
$\ov{W'}=\ov{\Phi^{-1}(W)}\subset \Phi^{-1}(\ov{W})\subset U'.$
By Lemma \ref{l1}, there exists a smooth function $f_0:V\oplus E\to [0,1]$  such that $\supp f_0\subset U'$ and $f=1$ on $W'$.  Since $f_0$ is smooth and  the map $\pi$ is sc-smooth, the composition $f_0\circ \pi$ is sc-smooth. Hence,  putting
 $f:=f_0\vert K$,  we obtain an sc-smooth function  defined on $K$  having support in $U$ and equal to $1$ on $W$.
\end{proof}

We will make use several times of the  following lemma from \cite{HWZ3}.

\begin{lem}\label{lem7.9.3}
Let $X$ be an ep-groupoid. For every point $x_0\in X$, there exists an open
neighborhood $U=U(x_0)\subset X$  having the following properties.
\begin{itemize}
\item[$\bullet$]  The natural representation of the isotropy group $G_{x_0}$ acts  by sc-diffeomorphisms on $ U$.
\item[$\bullet$] If $y\in X$ and there are no morphisms between $y$ and objects in
$\ov{U}$, then there exists an open neighborhood $V$ of $y$ so
that for every $z\in V$ there is no morphism to an object  in $\ov{U}$.
\item[$\bullet$] If $y\in X$ and  there are no morphisms  between $y$ and objects in
$U$, but there is a morphism to some element in $\partial U$, then given
an open neighborhood $W$ of $\partial U$ (the set theoretic boundary of $U$),
there exists an open neighborhood $V$ of $y$ so that if there is a morphism
$y'\to x'$ for some $y'\in V$ and $x'\in U$, then $x'\in W$.
\end{itemize}
The open set  $U$ can be taken as small as we wish.
\end{lem}

 If $g:X\to [0,1]$ is an sc-smooth function on the ep-groupoid $X$, we denote by $\abs{g}$  the continuous function defined on the orbit space $\abs{X}$ by $\abs{g}( \abs{ x}):=g(x)$.  Now we come to the statement of the main theorem of this section.

\begin{thm} [{\bf  sc-smooth partition of unity}] \label{scpounity}
Let $X$ be an  ep-groupoid  and let ${\mathcal O}=(O_{\alpha})_{\alpha \in A}$ be an  open cover of the orbit space  $\abs{X}$. Then there exists an sc-smooth partition of unity $(g_{\alpha})_{\alpha\in A}$ on  $X$ so that the associated continuous  partition of unity $(\abs{g_{\alpha}})_{\alpha\in A}$ of $\abs{X}$ is subordinate to ${\mathcal O}=(O_{\alpha})_{\alpha \in A}$.
\end{thm}

The proof of Theorem \ref{scpounity}  follows from the next  two lemmata.

\begin{lem}\label{invarinatcover}
Let ${\mathcal O}=(O_{\alpha})_{\alpha\in A}$ be an open cover of $\abs{X}$. Then  there exist locally finite open covers $(W_{j})_{j\in J}$ and $(U_{j})_{\beta \in J}$  subordinate to ${\mathcal O}$ and such that $\ov{W_j}\subset U_{j}$. The sets
$W_{j}$ and $U_{j}$ are invariant with respect to  the natural group action on
$U_{j}$, and the open cover $(\pi^{-1}(\pi (U_{j}))_{j\in J}$ is locally finite.
\end{lem}
\begin{proof}
In view of the paracompactness of $\abs{X}$,  there is  a locally finite refinement $(Q_{\alpha})_{\alpha\in A}$  of the cover $(O_{\alpha})_{\alpha \in A}$.  Then $(\pi^{-1}(Q_{\alpha}))_{\alpha \in A}$ is a locally finite refinement of $(\pi^{-1}(O_{\alpha}))_{\alpha \in A}$. For every point $x\in X$,  we choose an open neighborhood $V_x$ intersecting only a finite number of sets $\pi^{-1}(Q_{\alpha})$.  We replace  $V_x$ by its  intersection  with  those  $\pi^{-1}(Q_{\alpha})$  which contains  the point $x$. Observe that  there is no morphism between  the point $x$ and the sets $\pi^{-1}(Q_{\alpha})$ which don't intersect $V_x$.  Hence, shrinking $V_x$ further,  we may assume that $V_x$ has the  properties  listed in Lemma \ref{lem7.9.3} and that there are no morphisms  between points in $V_x$ and points in the sets $\pi^{-1}(Q_{\alpha})$  not intersecting  $V_x$.  The collection $(V_x)_{x\in X}$ is an open cover of $X$ and since $X$ is paracompact, there exists a locally finite refinement  $(U'_{j})_{j\in J}$ of $(V_x)_{x\in X}$. For every  $j\in J$, choose  a  point $x(j)$ such that $U'_{j}\subset V_{x(j)}$.  We abbreviate by $G_{j}$ the isotropy group $G_{x(j)}$  acting on $V_{x(j)}$ by its natural representation.
We claim that $(\pi^{-1}(\pi (U_j')))_{j\in J}$ is a locally finite cover of $X$. Indeed, take $y\in X$. Then $y\in U_k'\subset V_{x(k)}$ for some $k\in J$. Since $(U_j')_{j\in J}$ is locally finite, there exists an open neighborhood $W_y$ of $y$ contained in $U'_k$ and intersecting only a finite number of the sets $U_j'$, say $U'_{j_1}, \ldots ,U'_{j_N}$. Hence $k=j_i$ for some $1\leq i\leq N$. Replacing $W_y$  by a smaller set,  we may assume that $W_y$ is $G_{k}$-invariant.  Assume that $z\in W_y\cap \pi^{-1}(\pi (U'_j))$ for some $j\neq j_1,\ldots , j_N$.  Then there is a morphism between some point $v\in V_j$ and $z=\varphi_g (v)$ for some $g\in G_j$. In view of the definition of $V_{x(k)}$,  we have $v\in V_{x(k)}$.
Hence there is $h\in G_{k}$ such that $v=\varphi_h (z)$, and since $z\in W_y$ and $W_y$ is $G_j$-invariant $v\in W_y$. Consequently, $W_y\cap U'_j\neq \emptyset$ and it follows that $W_y$ intersects only the sets $\pi^{-1}(\pi (U'_{j_1})), \ldots ,\pi^{-1}(\pi (U'_{j_N}))$.
For every $j\in J$, set  $U_j=\bigcup_{g\in G_j}\varphi_g (U'_j)$. Then $\pi^{-1}(\pi (U_j))=\pi^{-1}(\pi (U_j'))$ and since the isotropy groups $G_j$ are finite, it follows that $(U_j)_{j\in J}$ is a locally finite cover of $X$ such that $U_j\subset V_{x(j)}$.  Using paracompactness of $X$ again, we find a locally finite cover $(W'_{j})_{j\in J}$ such that $\ov{W_j'}\subset U_{j}$.
Define  $W_{j}=\bigcup_{g\in G_{j}}\varphi_g (W_j')$. Then  $W_{j}$ is  a $G_{j}$-invariant  open subset of $V_{j}$ such that  $\ov{W_j}\subset U_j$, and   the open cover  $(W_j )_{j\in J}$ is  locally finite.  This completes the  proof of the lemma..
\end{proof}

\begin{lem}\label{lemmascp1}
Let  $U=U(x_0)\subset X$ be an open neighborhood of $x_0$ with   the properties as listed in  Lemma \ref{lem7.9.3} and let  $\varphi:U\to K^{\mathcal S}$ be  a coordinate chart onto an open subset of the splicing core $K^{\mathcal S}$. Assume that  $W$ is a $G_{x_0}$-invariant open subset of $U$ such that $\ov{W}\subset U$. Then there exists an sc-functor $f:X\to [0,1]$ satisfying $f=1$ on $\ov{W}$ and  $\supp f\subset \pi^{-1}(\pi (U)).$
\end{lem}
\begin{proof}
We choose an open $G_{x_0}$-invariant set $V$ such that  $\ov{W}\subset V\subset \ov{V}\subset U$.  With the help of  Lemma \ref{l2} and the chart $\varphi:U\to K^{\mathcal S }$, we find an sc-smooth function $f_0: X\to [0,1]$  satisfying $\supp f_0\subset V$ and $f_0=1$ on $\ov{W}$.  Define
 the function $f_1$ on $U$  by
 $$f_1 (x)=\dfrac{1}{\sharp G_{x_0}}\sum_{g\in G_{x_0}}f_0 (\varphi_g (x)), \quad x\in U.$$
Then   $f_1$ is sc-smooth as  a  finite sum of sc-smooth functions, $0\leq  f_1 \leq 1$, and  $f_1 =1$  on $x\in \ov{W}$ since $\ov{W}$ is $G_{x_0}$-invariant.   Since $f_0 =0$ on $U\setminus \ov{V}$ and $U\setminus \ov{V}$ is $G_{x_0}$-invariant,  it follows that also $f_1=0$ on $U\setminus \ov{V}$.  In particular, $f_1=0$ on $\partial V$.

We extend $f_1$ to the function  $f:X\to [0,1]$ as follows.  If $x\in U$,  then $f(x):=f_1 (x).$ If  there exists a morphism between $x$ and some point $y\in U$, then  set $f(x):=f_1  (y).$ Finally, if there is no morphism between $x$ and a point in $U$, then set  $f(x):=0.$

Clearly, $f(x)=f(y)$ if there  is a morphism $h: x\to y$. Note also that $f (x)=0$ for $x\in \partial U$. Indeed, if there are no morphisms  between $x$ and points of $U$, then  by the definition of $f$, $f(x)=0$. If there exists a morphism between $x$ and $y\in U$, then by Lemma \ref{lem7.9.3}, the point $x$ belongs to $U\setminus \ov{V}$,  so that again $f (x)=f_1(y)=0$.

We already know that $f$ is sc-smooth on $U$. To show that it is sc-smooth on $X$, we
take   $x\in X\setminus U$ and consider the following cases.  If  there is no morphism between $x$ and a point in  $\ov{U}$, in particular, there is no morphism between $x$ and a point in $U$, then $f(x)=0$. By   part (b) of Lemma \ref{lem7.9.3},  $f=0$ on some  open neighborhood $U_x$ of $x$  and so $f$ is sc-smooth on $U_x$.

Next assume that  there exists a morphism $h: x\to y$  between the point $x$ and a point $y\in U$. According to the definition of $f$,  $f (x)=f_1 (y)$. We find two  open neighborhoods $U_x$ and $U_y$  of $x$ and $y$ such that  $U_y\subset U$  and   $t\circ s^{-1}:U_x\to U_y$ is an sc-diffeomorphism.  Then, $f =f_1 \circ t\circ s^{-1}$ on $U_x$ and since the right hand side is an sc-smooth function,  the function $f$ is sc-smooth on $U_x$.

In the last case, assume that  there is no morphism between $x$ and points of $U$ but there is a morphism between $x$ and some  point $y\in \partial U$. Then again we find open neighborhoods $U_x$ and $U_y$ of points $x$ and $y$ such that $t\circ s^{-1}:U_x\to U_y$ is an sc-diffemorphism. By  Lemma  \ref{lem7.9.3}, we may  take  these neighborhoods so small that the following holds. If there exists a morphism between a point $y'\in U_y$ and a point $z\in U$, then necessarily $z\in U\setminus W$.  At $y'$ we have $f(y')=0$ since $f_1=0$ on $U\setminus W$. If there are no morphism between $y'\in U_y$ and points in $U$, then $f (y')=0$. Hence $f$ is equal to $0$ on $U_y$ and  since $f\vert U_x=(f\vert U_y)\circ (t\circ s^{-1})$,  we conclude that $f$ is equal to $0$ on $U_x$.  So we  proved that the function $f$ is sc-smooth on $X$.

It remains to prove  that $\supp f\subset \pi^{-1}(\pi (U))$. At every point $x\not \in \pi^{-1}(\pi (U))$, $f(x)=0$.   Hence  it  is  enough to show that for every $x\in \partial \pi^{-1}(\pi (U))$ there exists a neighborhood $U_x$ of $x$ such that $f=0$ on $U_x$.  To see this, we prove that there is an  open neighborhood $U_x$ of $x$ such that if there is a morphism between $x'\in U_x$ and a point in $y\in U$,  then $y\in U\setminus W$. Otherwise, we find a sequence  $(x_n)$ converging to $x$, a sequence $(y_n)\subset W$, and a sequence  $(h_n)$ of morphisms $h_n:y_n\to x_n$.  Since the map $t:s^{-1}(\ov{U(x_0}))\to X$ is proper, there is  subsequence of the morphisms $(h_n)$ converging to the  morphism $h$. This implies that  the subsequence of $(y_n)$ converges to the point $y\in W$ and that $h:y\to x$,  contradicting the fact that there are no morphisms between points in $\partial \pi^{-1}(\pi (U))$ and points in  $U$.   Hence $f =0$ on $U_x$ and this proves that $\supp f \subset \pi^{-1}(\pi (U)).$
\end{proof}

\begin{proof}[Proof of Theorem \ref{scpounity}]
Let ${(O_{\alpha})}_{\alpha \in A}$ be an open cover of $\abs{X}$.  In view of  Lemma \ref{invarinatcover}, there are open covers $(W_j)_{j\in J}$ and $(U_j)_{j\in J}$ subordinate to $(\pi^{-1}(O_{\alpha}))$ such that $\ov{W_j}\subset U_j$. Moreover, the sets
$W_j$ and $U_j$ are invariant with respect to the natural representation of $U_j$ and the cover $(\pi^{-1}(\pi (U_j)))_{j\in J}$ is locally finite. By Lemma \ref{lemmascp1}, for every
$j\in J$, there is an sc-smooth functor $f_j':X\to [0,1]$ which is equal to $1$ on $\ov{W_j}$ and $\supp f_j'\subset \pi^{-1}(\pi (U_j))$. Set $f'=\sum_{j\in J}f_j'$. In view of the local finiteness of   $(\pi^{-1}(\pi (U_j)))_{j\in J}$, the sum  has only a finitely many nonzero terms in a neighborhood of each point and thus defines an sc-smooth function. Because $f_j=1$ on $W_j$ and every point of $X$ is in some $W_j$, the sum is also positive. Now define
$f_j=\tfrac{f'_j}{f'}$. Then each $f_j$ is an sc-smooth functor such that $\supp f_j\subset \pi^{-1}(\pi (U_j))$.   Finally, we may reindex our functions $f_j$ so that they are indexed by the indices in the set $A$.  Since  the cover $(U_j)_{j\in J}$ is a refinement of  $(\pi^{-1}(O_{\alpha}))$,  we choose for each $j$ an index $\alpha (j)$ such that  $U_j\subset \pi^{-1}(O_{\alpha (j)})$. Then for each $\alpha\in A$,  we define $g_{\alpha}=\sum_{j, \,  \alpha (j)=\alpha}f_{j}$. If there is no $j$ such that $\alpha (j)=\alpha$, then we set $g_{\alpha}=0$.  Every $g_{\alpha}$ is smooth and invariant under morphisms and satisfies $0\leq g_{\alpha}\leq 1$ an $\supp g_{\alpha}\subset \pi^{-1}(O_{\alpha})$.  In addition, $\sum_{\alpha\in A}g_{\alpha}=\sum_{j\in J}f_j=1$.  Consequently, $(g_\alpha )_{\alpha \in A}$ is a desired sc-smooth partition of unity.
\end{proof}

\subsection{Poincar\'e Lemma for Sc-Differential Forms}
We begin with a basic version of the Poincar\'e lemma in an sc-context. We follow the proof in \cite{Lang} for the Hilbert-space case and  verify that all steps are valid in the sc-world. Then the  extension to ep-groupoids is fairly straightforward.

\begin{lem}\label{poncare0}
Assume that $U$ is a relatively open convex neighborhood of $0$ in the partial quadrant $[0,\infty)^n\oplus W$ where $W$ is an  sc-Banach space. Let $\omega$ be a closed sc-differential $k$-form defined on $U$. Then there exists a sc-differential form $\tau$ on $U$ so that $d\tau=\omega$ on $U^1$.
\end{lem}
\begin{proof}
For $v_1,\ldots ,v_{k-1}\in \R^n\oplus W$ and $y\in U$, we define  the form $\tau$ by setting
$$
\tau(y)(v_1,\ldots ,v_{k-1})=\int_0^1 t^{k-1}\omega(ty)(y,v_1,\ldots ,v_{k-1})dt.
$$
This is a well-defined expression. Using our  assumptions,   we see that  the map
$[0,1]\oplus U^1\oplus (\oplus_{k-1}({\mathbb R}^n\oplus W))\rightarrow {\mathbb R}$ defined by
$$
(t,y,v_1,\ldots ,v_{k-1})\mapsto \omega(ty)(y,v_1,\ldots, v_{k-1})
$$
is sc-smooth. Because $\tau$ is obtained by integrating with respect to $t$ over $[0,1]$,  we conclude  that
$$
\tau:U^1\oplus(\oplus_{k-1}({\mathbb R}^n\oplus W))\rightarrow {\mathbb R}
$$
is sc-smooth and hence an sc- differential $(k-1)$-form on $U$. Its exterior derivative $d\tau$ is an sc-differential $k$-form on $U^1$. In order to show that $d\tau=\omega$,  it suffices to show the equality on a dense set of points in $U$. Take a smooth point $y\in U$ and   smooth points $e_0,\ldots, e_{k-1}\in  \R^n\oplus W$  which we view as constant vector fields. Then we make use of the fact that the form $\omega$ is closed and  find, for $t\in [0,1],$ that
\begin{equation}\label{opera}
\begin{split}
0&=d\omega(ty)(y,e_0,\ldots ,e_{k-1})
= \dfrac{d}{dt}(\omega(ty)(e_0,\ldots ,e_{k-1}))\\
&\phantom{quad }-
\sum_{i=0}^{k-1} (-1)^i (D\omega(ty)\cdot e_i)(y,e_0,\ldots ,\wh{e_i},\ldots ,e_{k-1}.
\end{split}
\end{equation}
Using (\ref{opera}) we  find that  at a smooth point $y\in U$,
\begin{equation*}
\begin{split}
&d\tau(y)(e_0,\ldots ,e_{k-1})\\
&= \sum_{i=0}^{k-1} (-1)^i \int_0^1 t^{k-1} ( D(\omega(ty)(y,e_0,\ldots ,\wh{e_i},\ldots ,e_{k-1}))\cdot e_i)dt\\
&= \sum_{i=0}^{k-1} (-1)^i \int_0^1 t^{k-1}\omega(ty)(e_i,e_0,\ldots ,\wh{e_i},\ldots ,e_{k-1}) dt\\
&\phantom{=}+ \sum_{i=0}^{k-1} (-1)^i \int_0^1 t^{k-1}(D\omega(ty)\cdot e_i)(y,e_0,\ldots ,\wh{e_i},\ldots ,e_{k-1})dt\\
&=\int_0^1 k t^{k-1} \omega(ty)(e_0,\ldots ,e_{k-1}) dt
+\int_0^1 t^{k-1}\frac{d}{dt}(\omega(ty)(e_0,\ldots ,e_{k-1})) dt\\
&=\int_0^1 \frac{d}{dt}(t^k\omega(ty)(e_0,\ldots ,e_{k-1}))dt= \omega(y)(e_0,\ldots ,e_{k-1}).
\end{split}
\end{equation*}
This shows that $d\tau=\omega$ on $U^1$ and completes the proof.
\end{proof}
The argument in the  proof of Lemma \ref{poncare0} can be slightly modified to hold near points on level $1$. This time we cannot assume that the point in question is $0=(0,0)$ but  is of  the form $(0,w_0)$ with $w_0$ on level $1$.
Assume that $K$ is a splicing core associated to the splicing $(\pi,E,V)$
where $V$ is relatively open in $[0,\infty)^n\oplus W$. Let $x\in K$ be smooth
and  let $\omega$  be an sc-differential $k$-form defined on an open neighborhood of $x$.
Suppose that the form $\omega$ is closed. Then we show that $\omega=d\tau$ on some open neighborhood of $x$ in level $1$.
The map $r:V\oplus E\rightarrow K:(v,e)\rightarrow (v,\pi_v(e))$ pulls back the form  $\omega$ to a form on some relatively open set in a partial quadrant. We may assume without loss of generality that the latter is defined around the point $0$. By the previous lemma,  we find an sc-differential $(k-1)$-form  $\tau$ defined near $0$ so that $d\tau=r^\ast\omega$ on $U^1$ where $U$ is an open neighborhood of $x$ on level $0$. Using the identity $r\circ r=r$,  we conclude that
$$
d(r^\ast\tau) = r^\ast\omega.
$$
Restricting to the splicing core implies that if  $O$ is an open neighborhood of a smooth point, then  $d\tau=\omega$ on $O^1$.

Next we assume that we are in the ep-groupoid world. Let $X$ be an ep-groupoid and let $x\in X$ be a smooth point. Take an open neighborhood $U(x)$ which is homeomorphic to an open subset of splicing core $K$ and  which  is invariant under the  natural representation of the isotropy group $G_x$.  Let  $\omega$  be a closed form on $U(x)$  which is compatible with the $G_x$-action.  We  push-forward  the form $\omega$  to an open open set in the splicing core $K$ and use  Lemma \ref{poncare0} to construct a form $\tau$ so that $d\tau=\omega$. We pull-back this form to $U(x)$ and average it to obtain
a form $\tau'$ satisfying  $d\tau'=\omega$. Using morphisms we can extend the definition of  $\tau'$ to the saturated open neighborhood of $\pi^{-1}(\pi (U(x)))$ of the point $x$. (Here $\pi$ is the quotient map $\pi:X\to \abs{X}$.)

Hence we obtain the following  result.
\begin{prop}
Let $X$ be  an ep-groupoid and  let $\omega$  be a closed sc-differential $k$-form on $X$. Then, given a  point $x\in X_\infty$,  there exists a saturated open neighborhood $U$ of $x$ in $X$  and an sc- differential $(k-1)$-form $\tau$ on $U$ so that
$$
d\tau = \omega\quad \text{on}\ U^1.
$$
\end{prop}


\begin{thebibliography}{99}
\bibitem{BEHWZ} F.~Bourgeois, Y.~Eliashberg, H.~Hofer,
K.~Wysocki and E.~Zehnder,  Compactness Results in Symplectic Field
Theory, {\em Geometry and Topology}, Vol. 7, 2003, pp.799-888.
\bibitem{CRS} K. Cieliebak, I. Mundet i Riera\ and\ D. A. Salamon,
 Equivariant moduli problems, branched manifolds,
 and the Euler class,
Topology {\bf 42} (2003), no.~3, 641--700.
\bibitem{EGH} Y. Eliashberg, A. Givental\ and\ H. Hofer, Introduction to Symplectic Field Theory,
 Geom. Funct. Anal. {\bf 2000}, Special Volume, Part II, 560--673.
 \bibitem{Fathi}A.~ Fathi, Partitions of Unity for Countable Covers. {\em  The American Math-
ematical Monthly}, Vol. 104, No. 8. (Oct., 1997), pp. 720-723.
\bibitem{G} M.~Gromov, Pseudoholomorphic Curves in
Symplectic Geometry, {\it Inv. Math.} Vol. 82 (1985), 307-347.
\bibitem{Hofer} H. Hofer, A General Fredholm Theory and
Applications, Current Developments in Mathematics, edited by D.
Jerison, B. Mazur, T. Mrowka, W. Schmid, R. Stanley, and S. T. Yau,
International Press, 2006.
\bibitem{HWZ1} H. Hofer, K. Wysocki\ and\ E. Zehnder,
A General Fredholm Theory {I}: A Splicing-Based Differential
Geometry, JEMS, Volume 9, Issue 4 (2007),  841--876.
\bibitem{HWZ2} H. Hofer, K. Wysocki\ and\ E. Zehnder,
A General Fredholm Theory {II}: Implicit Function Theorems,
preprint.
\bibitem{HWZ3} H. Hofer, K. Wysocki\ and\ E. Zehnder,
A General Fredholm Theory {III}: Fredholm Functors and Polyfolds,
preprint.
\bibitem{HWZ6} H. Hofer, K. Wysocki\ and\ E. Zehnder,
A General Fredholm Theory {IV}: Operations , paper in preparation.
\bibitem{HWZ8} H. Hofer, K. Wysocki\ and\ E. Zehnder,
Connections and Determinant Bundles for Polyfold Fredholm Operators,
paper in preparation.
\bibitem{HWZ4} H. Hofer, K. Wysocki\ and\ E. Zehnder,
Applications of Polyfold Theory {I}: Gromov-Witten Theory, paper in
preparation.
\bibitem{HWZ5} H. Hofer, K. Wysocki\ and\ E. Zehnder,
Applications of Polyfold Theory {II}: The Polyfolds of Symplectic
Field Theory, paper in preparation.
\bibitem{Lang2} S.~Lang, {\em Fundamentals of Differential Geometry}, 2nd ed.,  Springer, Vol.  191, 1999.
\bibitem{Lang} S.~Lang, {\em Differential and Riemannian Manifolds}, 3rd ed.,  Springer Graduate Texts , Vol.  160, 1995.
\bibitem{Mc} D. McDuff, Groupoids, Branched Manifolds and
Multisection, J. Symplectic Geom. 4, 259-315 (2006).
%\bibitem{MS} D. McDuff \ and\ D. Salamon, {\it Introduction
% to symplectic topology}, 2nd edition, Oxford University Press, 1998.
\bibitem{Mj} I. Moerdijk, Orbifolds as Groupoids: An Introduction, Contemp. Math. 310, 205-222 (2002).
\bibitem{MM} I. Moerdijk and J. Mr\v cun, {\it Introduction to
Foliation and Lie Groupoids}, Cambridge studies in advanced
mathematics, Vol. 91, 2003.


\end{thebibliography}
\end{document}